\documentclass[a4paper,12pt]{article}
\usepackage{amsmath, amsfonts, amsthm, graphicx, amssymb, url, enumitem,subcaption}
\usepackage[margin=2cm]{geometry}
\usepackage[section]{placeins}
\usepackage{hyperref}
\hypersetup{
    colorlinks=true,
    linkcolor=black,
    urlcolor=black,
    citecolor=black,
    linktoc=all
}
\newtheorem{cor}{Corollary}
\newtheorem{thm}{Theorem}
\newtheorem{lem}{Lemma}
\newtheorem{prop}{Proposition}  
\newtheorem{defn}{Definition}

\theoremstyle{definition}

\newcommand{\R}{\mathbb R}
\newcommand{\Z}{\mathbb Z}
\newcommand{\N}{\mathbb N}
\renewcommand{\P}{\mathbb{P}^2}
\newcommand{\Pn}{\mathbb{P}^n}
\renewcommand{\a}{\alpha}
\renewcommand{\b}{\beta}

\renewcommand{\c}{\gamma}
\newcommand{\conic}[1]{\mathcal{C} (#1)}

\renewcommand{\l}{\lambda}

\title{\Large{\textbf{Nets of lines with the combinatorics of the square grid and with touching inscribed conics}}}

\author{Alexander I. Bobenko\thanks{Partially supported by the DFG Collaborative Research Center TRR 109 ``Discretization in Geometry and Dynamics"} \; and Alexander Y. Fairley\thanks{Supported by a Leverhulme Trust Study Abroad Studentship} \\[6pt] Institut f{\"u}r Mathematik, Technische Universit{\"a}t Berlin,\\ Str. des 17. Juni 136, 10623 Berlin, Germany \\[6pt] bobenko@math.tu-berlin.de \; fairley@math.tu-berlin.de }

\date{18 November 2019}

\begin{document}

\sloppy
\maketitle

\begin{abstract}

In the projective plane, we consider congruences of straight lines with the combinatorics of the square grid and with all elementary quadrilaterals possessing touching inscribed conics. The inscribed conics of two combinatorially neighbouring quadrilaterals have the same touching point on their common edge-line. We suggest that these nets are a natural projective generalisation of incircular nets. It is shown that these nets are planar Koenigs nets. Moreover, we show that general Koenigs nets are characterised by the existence of a 1-parameter family of touching inscribed conics. It is shown that the lines of any grid of quadrilaterals with touching inscribed conics are tangent to a common conic. These grids can be constructed via polygonal chains that are inscribed in conics. The special case of billiards in conics corresponds to  incircular nets.
\end{abstract}

\newpage

\section{Introduction}\label{section: introduction}

The geometry of incircular nets (IC-nets) has recently been discussed in great detail in \cite{IC}. IC-nets were introduced by B\"ohm \cite{Bohm} and they are defined as congruences of straight lines in the plane with the combinatorics of the square grid such that each elementary quadrilateral admits an inscribed circle. IC-nets have a wealth of geometric properties, including the distinctive feature that any IC-net comes with a conic to which the gridlines are tangent. IC-nets are closely related to Poncelet(-Darboux) grids, which were originally introduced by Darboux \cite{D8789} and further studied in \cite{LeviTabachnikov} and \cite{Schwartz}.

Checkerboard IC-nets constitute a natural generalisation of IC-nets. The gridlines of checkerboard IC-nets have the combinatorics of the square grid but it is only required that every second quadrilateral admits an inscribed circle, namely the ``black'' (or ``white'' if the colours are interchanged) quadrilaterals if the quadrilaterals of the net are combinatorially coloured like those of a checkerboard. Checkerboard IC-nets can be consistently oriented so that their lines and circles are in oriented contact. Thus, these nets are naturally treated in terms of Laguerre geometry. In \cite{BST} checkerboard IC-nets were explicitly integrated in terms of Jacobi elliptic functions similar to that of elliptic billiards \cite{DR}. Recently in \cite{BLPT} the corresponding definitions and results were extended to the cases of incircular nets in the 2-sphere and also in the hyperbolic plane by developing the corresponding Laguerre geometries.

In this paper we suggest a purely projective generalisation of IC-nets. Namely, we consider planar congruences of straight lines with the combinatorics of the square grid and with all elementary quadrilaterals possessing touching inscribed conics (see Figure~\ref{figure: maintheoremproof}). It is worth mentioning that the lines of the projective grids we introduce correspond not to the lines of IC-nets but to the lines passing through the centres of their circles. We describe their geometry in detail and show, in particular, in Section \ref{subsection: sixlinestangenttoaconic} that the lines of these grids touch a common conic. A further important property is that planar grids of quadrilaterals with touching inscribed conics are planar Koenigs nets. Koenigs nets are an important example of integrable discrete differential geometry \cite{DDG}. In Section \ref{subsection: koenigsnets}, we show that the property to possess touching inscribed conics is characteristic for general Koenigs net. This characterisation of Koenigs nets via inscribed conics (Theorem~\ref{thm: Koenigs}) was independently discovered by Christian M\"uller. 

Our geometric analysis is based essentially on Theorem~\ref{thm: mainthm}, which is an incidence theorem for five conics and six touching lines, see Figure~\ref{figure: maintheoremproof}. The corresponding implications for grids of quadrilaterals with touching inscribed conics are described in Section~\ref{section: grids}. In particular, it is shown that these grids can be constructed via polygonal chains that are inscribed in conics. In Section \ref{subsection: billiards} it is demonstrated how the special case of billiards in conics can be used to generate incircular nets.

\section{Preliminaries}\label{section: lemma}

In this section we present some known results about inscribed conics. Many theorems about quadrilaterals with inscribed conics can be found in, for instance, the chapters XII, XVI and XVIII of \cite{Chasles}. Many other theorems about conics can be found in \cite{Universe}.

We denote by $u \in \R^{n+1}$ the homogeneous coordinates of the corresponding points $\{[u]: u \in \R^{n+1}\}$ of the projective space $\Pn$. In the projective plane $\P$, any arrangement of lines is called \emph{generic} if and only if no three of the lines are concurrent. Let $\Box([u], [v], [w], [x])$ denote the quadrilateral with the vertices $[u], [v], [w], [x]$ and with the generic edge-lines $([u], [v]), ([v],[w]), ([w],[x]), ([x],[u])$. The lines $([u], [w])$ and $([v], [x])$ are the diagonals of the quadrilateral.

\begin{defn}\label{defn: nondeginconic}
For any quadrilateral $\Box([u], [v], [w], [x])$ in $\P$, a \emph{non-degenerate inscribed conic} is a non-degenerate conic $\conic{\phi} := \{[x] \in P : \phi(x,x)=0\}$ defined by a non-zero symmetric bilinear form $\phi$ such that
\begin{align*}
\phi(u, \a_1u +\b_1 v) & = 0 = \phi(\a_1u +\b_1 v, v)\\
\phi(v, \a_2 v +\b_2 w) & =  0 =  \phi(\a_2 v +\b_2 w, w)\\
\phi(w, \a_3 w +\b_3 x) & =  0 =   \phi(\a_3 w +\b_3 x, x)\\
\phi(x, \a_ 4 x+\b_4 u) & =  0 =  \phi(\a_ 4 x+\b_4 u, u)
\end{align*}
where $[\a_1u +\b_1 v]$, $[\a_2 v +\b_2 w]$, $[\a_3 w +\b_3 x]$, $[\a_ 4 x+\b_4 u]$ are points, which are contained in the edge-lines, that are distinct from the vertices of the quadrilateral.
\end{defn}

In other words, the four edge-lines of the quadrilateral are the polars of the four points  $[\a_1u +\b_1 v]$, $[\a_2 v +\b_2 w]$, $[\a_3 w +\b_3 x]$, $[\a_ 4 x+\b_4 u]$ which are contained in the conic $\conic{\phi}$.

Proposition~\ref{prop: opptangencypts} is a degenerate case of Brianchon's theorem \cite{BergerII, Universe}.

\begin{prop}\label{prop: opptangencypts}
Let $\mathcal{C}$ be a non-degenerate conic that is inscribed in a quadrilateral $\Box([u], [v], [w], [x])$ in $\P$ and let $[\a_1u +\b_1 v]$, $[\a_2 v +\b_2 w]$, $[\a_3 w +\b_3 x]$, $[\a_ 4 x+\b_4 u]$ be the four tangency points. Then, the lines  $([\a_1u +\b_1 v], [\a_3 w +\b_3 x])$ and $([\a_2 v +\b_2 w],[\a_ 4 x+\b_4 u])$ are concurrent with the two diagonals of the quadrilateral. (See Figure~\ref{figure: mainlemma}.)
\end{prop}

Consider a triangle $\triangle([u], [v], [w])$. Let $[\a_1u+\b_1v]$, $[\a_2v+\b_2w]$, $[\a_3w + \b_3u]$ be distinct from the vertices. The points form a \emph{Ceva configuration} if and only if the three lines $([u],[\a_2v+\b_2w])$, $([v],[\a_3w+\b_3u])$ and $([w],[\a_1u+\b_1v])$ are concurrent. The points form a \emph{Menelaus configuration} if and only if the three points $[\a_1u+\b_1v]$, $[\a_2v+\b_2w]$, $[\a_3w + \b_3u]$ are collinear.

\begin{figure}[htbp]
\begin{center}
  \begin{subfigure}[t]{0.4\textwidth}
    \includegraphics[width=\textwidth]{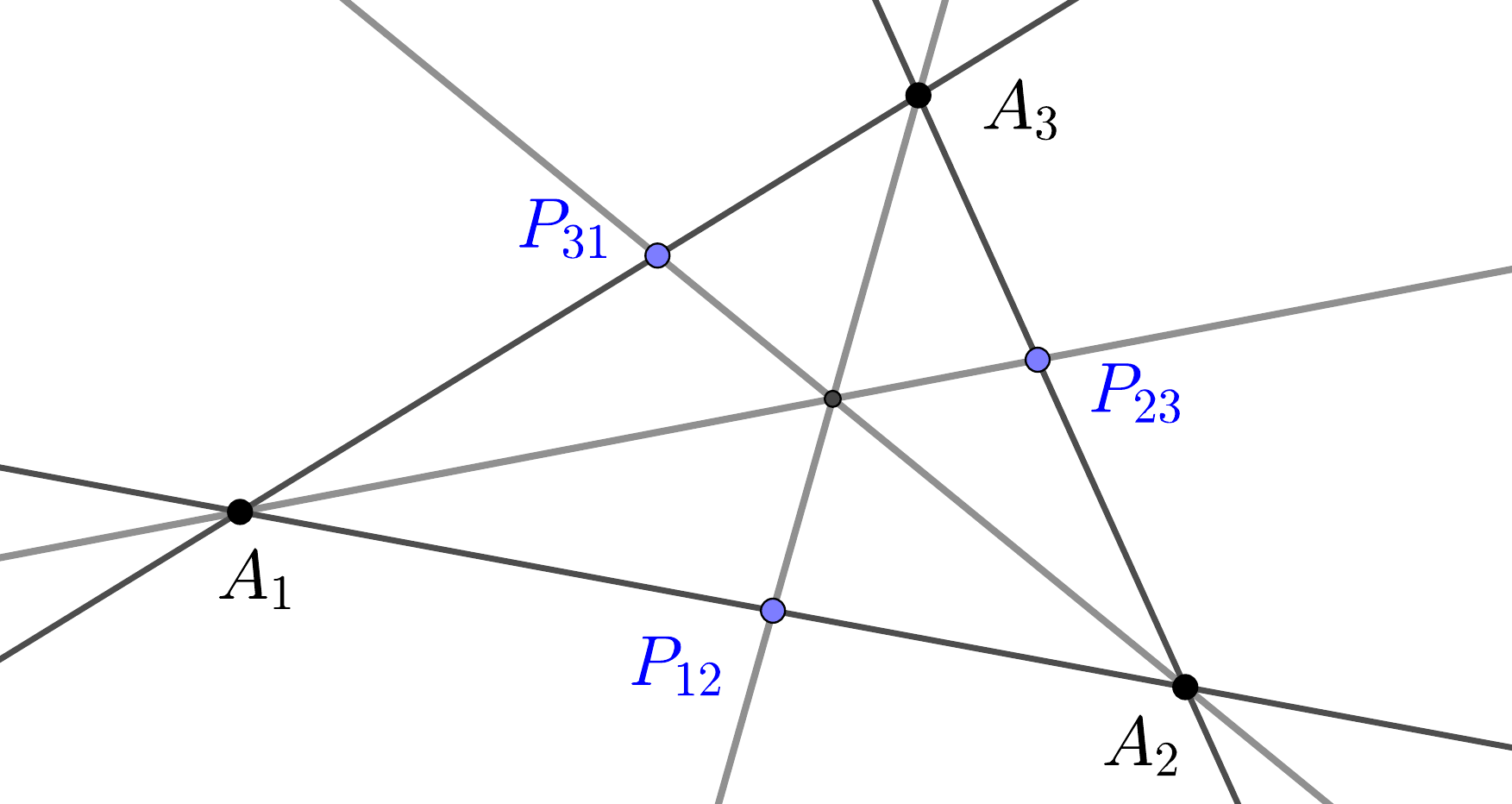}
 \end{subfigure}
  \hspace{0.5cm}%
  \begin{subfigure}[t]{0.4\textwidth}
    \includegraphics[width=\textwidth]{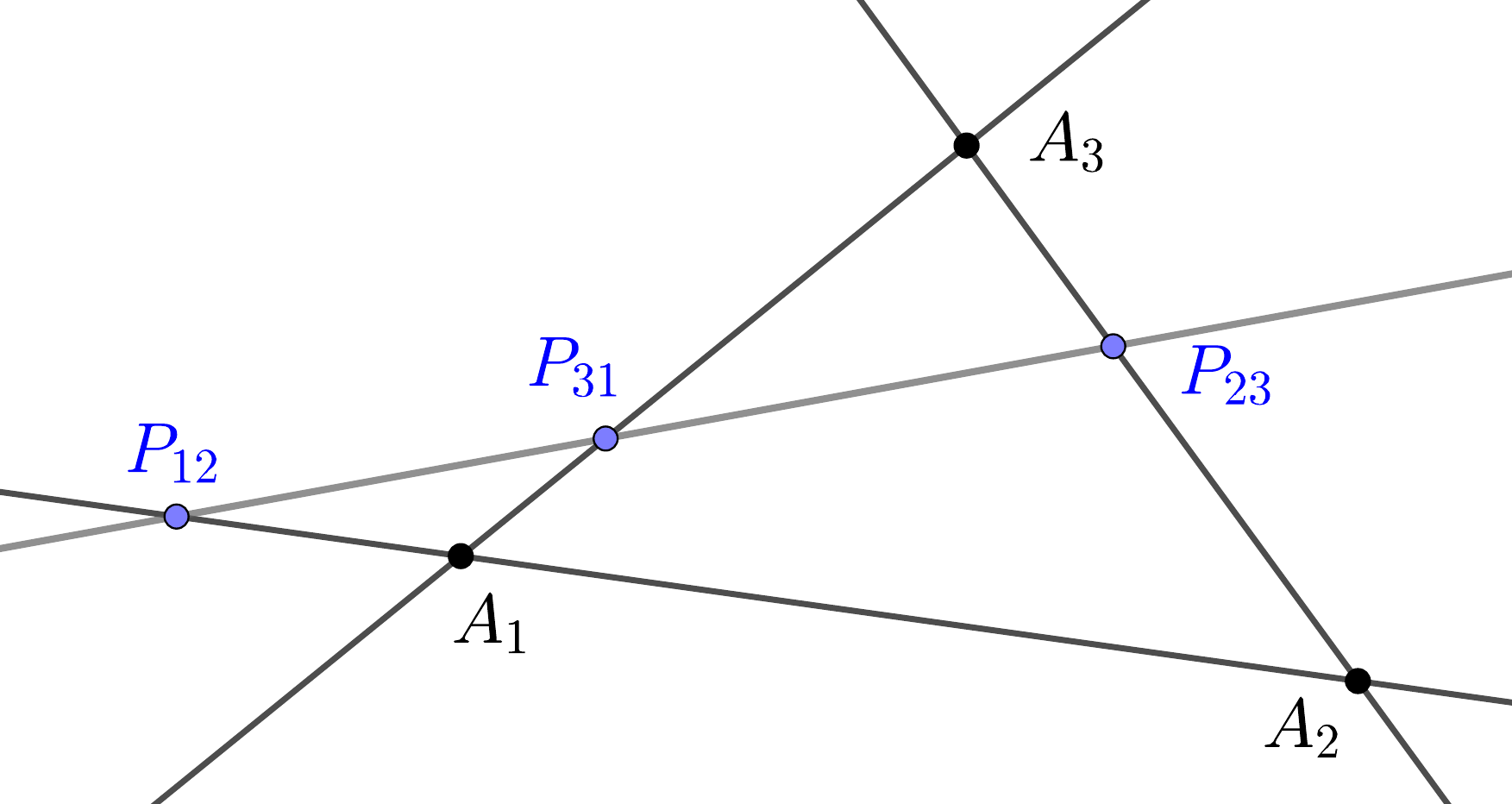}
  \end{subfigure}
  \end{center}
\caption{A Ceva configuration and a Menelaus configuration.}
\end{figure}
 
\begin{thm}[Ceva's theorem and Menelaus' theorem]\label{thm: cevamenelaus}
Consider a triangle $\triangle(A_1, A_2, A_3)$ in the affine plane. Let $P_{12}$, $P_{23}$ and $P_{31}$ be points on the respective edge-lines $(A_1,A_2)$, $(A_2,A_3)$ and $(A_3,A_1)$ distinct from the vertices of the triangle. Then,
\begin{enumerate}[label=(\roman*)]
\item $\frac{l(A_1,P_{12})}{l(P_{12},A_2)} \cdot  \frac{l(A_2,P_{23})}{l(P_{23},A_3)} \cdot \frac{l(A_3,P_{31})}{l(P_{31},A_1)} = 1$ if and only if the points form a Ceva configuration.
\item $\frac{l(A_1,P_{12})}{l(P_{12},A_2)} \cdot  \frac{l(A_2,P_{23})}{l(P_{23},A_3)} \cdot \frac{l(A_3,P_{31})}{l(P_{31},A_1)} = -1$ if and only if the points form a Menelaus configuration.
\end{enumerate}
Here, $l(A,P)$ denotes an oriented length.
\end{thm}

Note that the quotient of the oriented lengths is invariant with respect to the line orientation. Theorem~\ref{thm: cevamenelaus} can be found, for example, in \cite{DDG, Perspectives}.

\begin{defn}
Let $A_1, P_{12}, A_2, P_{21}$ be collinear points in affine space $\mathbb{A}^n$, $n\geq 1$. The \emph{cross ratio} is defined by $${\rm{cr}}(A_1, P_{12}, A_2, P_{21}):=\frac{l(A_1,P_{12})}{l(P_{12},A_2)}\frac{l(A_2,P_{21})}{l(P_{21},A_1)}.$$
\end{defn}

If ${\rm{cr}}(A_1, P_{12}, A_2, P_{21})=-1$, then the cross ratio is called \emph{harmonic} and the point $P_{21}$ is called the \emph{harmonic conjugate} of $P_{12}$ with respect to $A_1$ and $A_2$. Proposition~\ref{prop: cevamenelausharmonic} provides a well known method to construct the harmonic conjugate of $P_{12}$ with respect to $A_1$ and $A_2$.

\begin{prop}\label{prop: cevamenelausharmonic}
Let $[\a_1 u + \b_1 v]$, $[\a_2 v + \b_2 w]$, $[\a_3w + \b_3 u]$ be points that form a Ceva configuration on the triangle $\triangle([u],[v],[w])$. Then, the line $([\a_2 v + \b_2 w], [\a_3w + \b_3 u])$ intersects the line $([u],[v])$ at the harmonic conjugate of $[\a_1 u + \b_1 v]$ with respect to $[u]$ and $[v]$.
\end{prop}

\begin{figure}[htbp]
\[\includegraphics[width=0.8\textwidth]{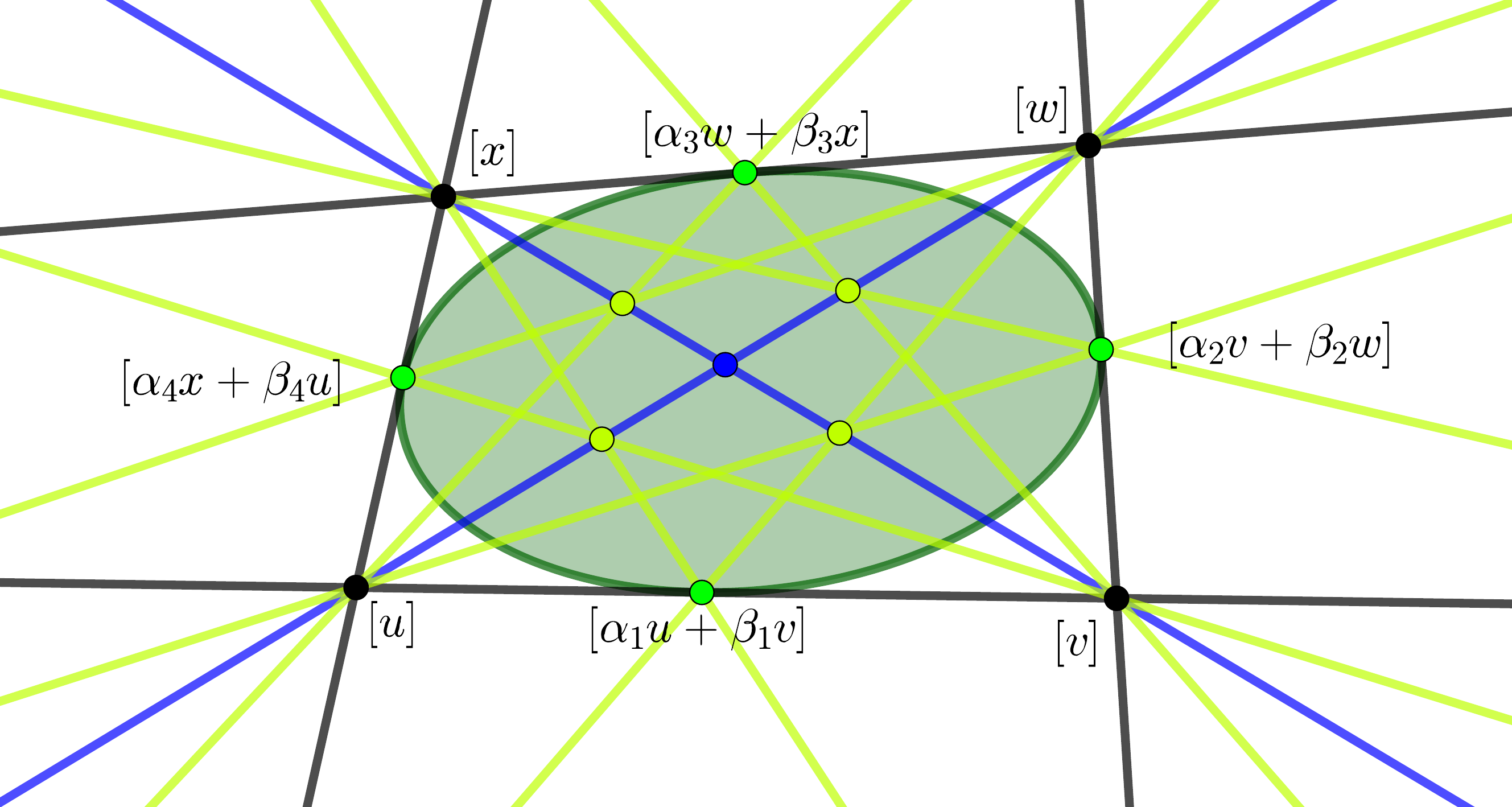}\]
\caption{For any non-degenerate conic that is inscribed in a quadrilateral $\Box([u], [v], [w], [x])$, the two lines connecting the opposite tangency points are concurrent with the two diagonals. The tangency points determine a Ceva configuration on each of the triangles $\triangle([u], [v], [w])$, $\triangle ([v], [w], [x])$, $\triangle ([w], [x], [u])$, $\triangle ([x], [u], [v])$. The intersection point of the diagonals of $\Box([u], [v], [w], [x])$ is a point of the four Ceva configurations.}
\label{figure: mainlemma}
\end{figure}

\begin{figure}[htbp] 
\[\includegraphics[width=0.8\textwidth]{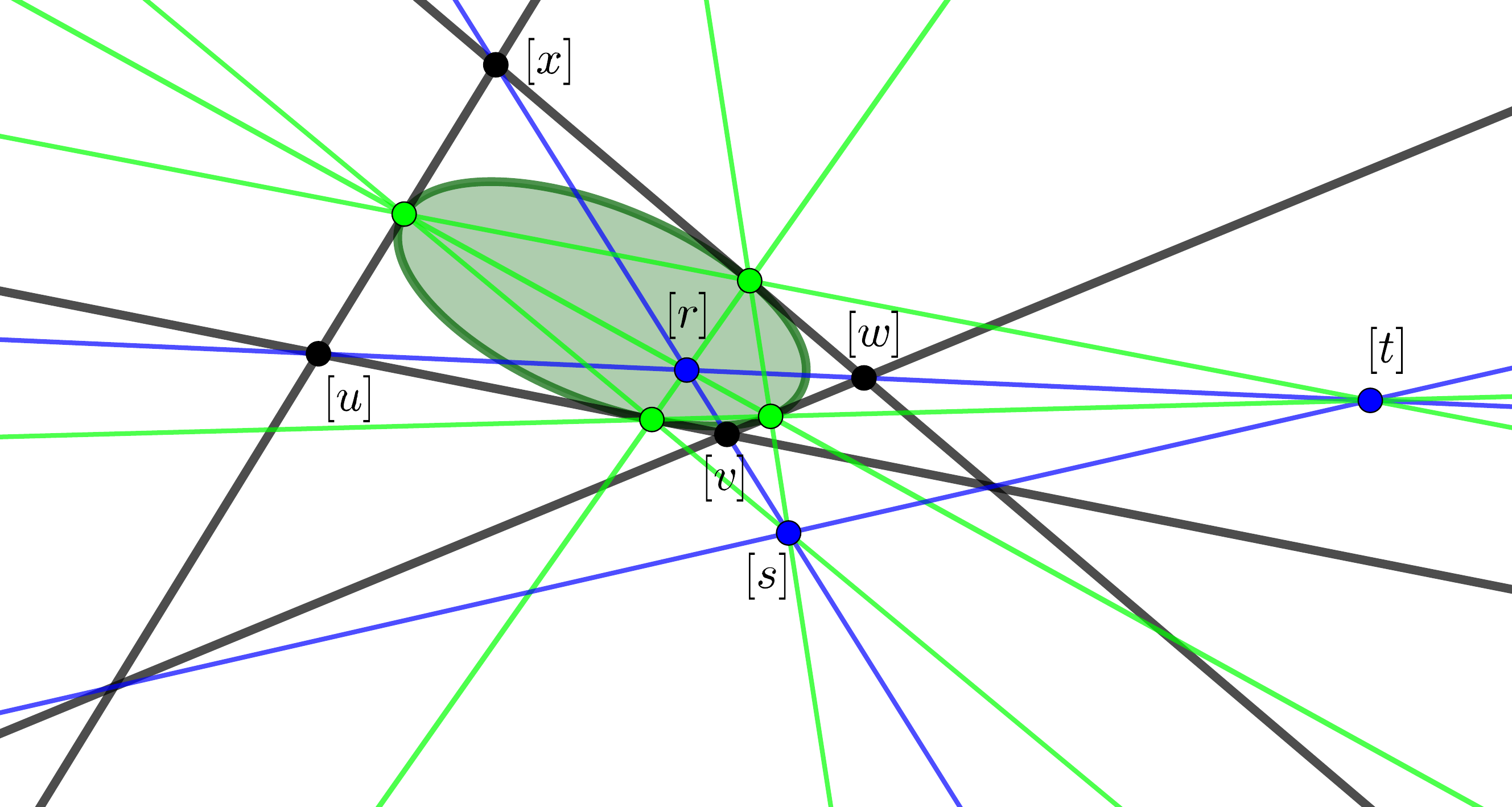}\]
\caption{By the complete quadrilateral theorem, ${\rm{cr}}([u], [r] , [w], [t]) = -1$ and ${\rm{cr}}([v], [r], [x], [s]) = -1$.}
\label{figure: proofoflemma}
\end{figure}

\begin{thm}\label{thm: inquad}
Let $[\a_1u +\b_1 v]$, $[\a_2 v +\b_2 w]$, $[\a_3 w +\b_3 x]$, $[\a_ 4 x+\b_4 u]$ be four distinct points that are distinct from the vertices of the quadrilateral $\Box([u], [v], [w], [x])$. These points determine a Ceva configuration for each of the triangles $\triangle([u], [v], [w])$, $\triangle ([v], [w], [x])$, $\triangle ([w], [x], [u])$ and $\triangle ([x], [u], [v])$. The intersection point of the diagonals of $\Box([u], [v], [w], [x])$ is a common point of the four Ceva configurations if and only if the points $[\a_1u +\b_1 v]$, $[\a_2 v +\b_2 w]$, $[\a_3 w +\b_3 x]$ and $[\a_ 4 x+\b_4 u]$ are the tangency points of a non-degenerate inscribed conic. (See Figure~\ref{figure: mainlemma}.)
\end{thm}

\begin{proof}
Suppose that the intersection point of the diagonals of $\Box([u], [v], [w], [x])$ is a common point of the four Ceva configurations. Then, the representative vectors $u, v, w, x$ for the vertices of the quadrilateral can be chosen so that $[u+v], [v+w], [w+x], [x+u]$ are the points on the edge-lines and so that $[u+w] = [v+x]$, which is the intersection point of the two diagonals. The identity $[u+w] = [v+x]$ implies that there exists some non-zero $\c \in \R$ such that $\c(u+w) = v+x$. Equivalently, $x = \c u -v + \c w$.

Define a non-zero symmetric bilinear form $\phi$ on the basis $u, v, w \in \R^3$ by the following system of equations.
\begin{align*}
\phi(u,u) &= \phi (v,v) = \phi (w,w) = -\phi(u,v) = - \phi (v,w) = 1  \\ 
\phi(u,w) &= -\frac{\c + 1+1}{\c} 
\end{align*}
By substituting $x= \c u - v+ \c w$,  the following equations can be verified.
\begin{align*}
\phi(u,u) = \phi(v,v) = \phi(w,w) = \phi (x,x) = -\phi(u,v) = -\phi (v,w) = - \phi(w,x) = -\phi (x,u)
\end{align*}
Equivalently,
\begin{align*}
\phi(u,u+v) & = 0 = \phi(v,u+v)\\
\phi(v,v+w) & = 0 = \phi(w,v+w)\\
\phi(w,w+x) & = 0 =\phi(x,w+x)\\
\phi(x,x+u) & = 0 = \phi(u , x + u)
\end{align*}
Looking for a contradiction, suppose that the conic $\conic{\phi}$ is degenerate. By the classification of degenerate conics, which can be found in \cite{BergerII}, $\conic{\phi}$ is either a double point, a double line or a pair of lines. Because $\conic{\phi}$ is tangent to the four generic edge-lines of the quadrilateral, it cannot be a double point nor a pair of lines. $\conic{\phi}$ must be a double line. Then, the points $[u+v], [v+w], [w+x], [x+u]$ are collinear and the double line $\conic{\phi}$ also contains the points $[u-w]$ and $[v-x]$. By the complete quadrilateral theorem \cite{BergerI}, using the fact that $[u+w]=[v+x]$ is the intersection of the diagonals,  the harmonic ratios ${\rm{cr}}([u, [u+w], [v], [u-w]) = -1$ and ${\rm{cr}}([v], [v+x], [x], [v-x])=-1$ imply that $[u-w]$ and $[v-x]$ are contained in the line $(([v],[w])\cap ([x],[u]), ([u], [v]) \cap ([w],[x]))$.  Then, $[u+v] = [w+x]$ and $[v+w] = [x+u]$. This contradicts the assumption that the four points on the edge-lines are distinct. Therefore, the conic $\conic{\phi}$ is a non-degenerate inscribed conic.

Suppose that $[\a_1u +\b_1 v]$, $[\a_2 v +\b_2 w]$, $[\a_3 w +\b_3 x]$, $[\a_ 4 x+\b_4 u]$ are the four tangency points of a non-degenerate inscribed conic $\mathcal{C}$. By Proposition~\ref{prop: opptangencypts}, the lines $([\a_1u +\b_1 v], [\a_3 w +\b_3 x])$ and $([\a_2 v +\b_2 w], [\a_ 4 x+\b_4 u])$ are concurrent, say at $[r]$, with the two diagonals of the quadrilateral. Similarly, the two lines $([\a_2u +\b_2 v], [\a_3 v +\b_3 w]), ([\a_4 w +\b_4 x] , [\a_ 1 x+\b_1 u])$ are also concurrent, say at [s], with two of the diagonals of the complete quadrilateral and the two lines $([\a_1u +\b_1 v], [\a_2 v +\b_2 w]), ([\a_3 w +\b_3 x], [\a_ 4 x+\b_4 u])$ are concurrent, say at $[t]$, with two of the diagonals of the complete quadrilateral. See Figure~\ref{figure: proofoflemma}. By the complete quadrilateral theorem \cite{BergerI}, ${\rm{cr}}([u], [r] , [w], [t]) = -1$ and ${\rm{cr}}([v], [r], [x], [s]) = -1$. By Theorem~\ref{thm: cevamenelaus}, $[t] = [-\b_1\b_2 w + \a_1\a_2 u]$ because $[\a_1 u+ \b_1 v]$, $[\a_2 v + \b_2 w]$ and $[-\b_1\b_2 w + \a_1\a_2 u]$ are the points of a Menelaus configuration on the triangle $\triangle([u], [v], [w])$. Then, the harmonic ratio ${\rm{cr}}([u], [r] , [w], [t]) = -1$ implies that $[r] = [\b_1\b_2 w + \a_1\a_2 u]$. Then, by Theorem~\ref{thm: cevamenelaus}, the points $[r]$, $[\a_1 u + \b_1v]$ and $[\a_2 v + \b_2 w]$ form a Ceva configuration on the triangle $\triangle([u], [v], [w])$. Similarly, the point $[r]$ is a point of the three other Ceva configurations. 
\end{proof}

Theorem~\ref{thm: inquad} is a generalisation of the fact that Ceva configurations correspond to non-degenerate conics that are inscribed in triangles.

\begin{figure}[htbp] 
\[\includegraphics[width=0.55\textwidth]{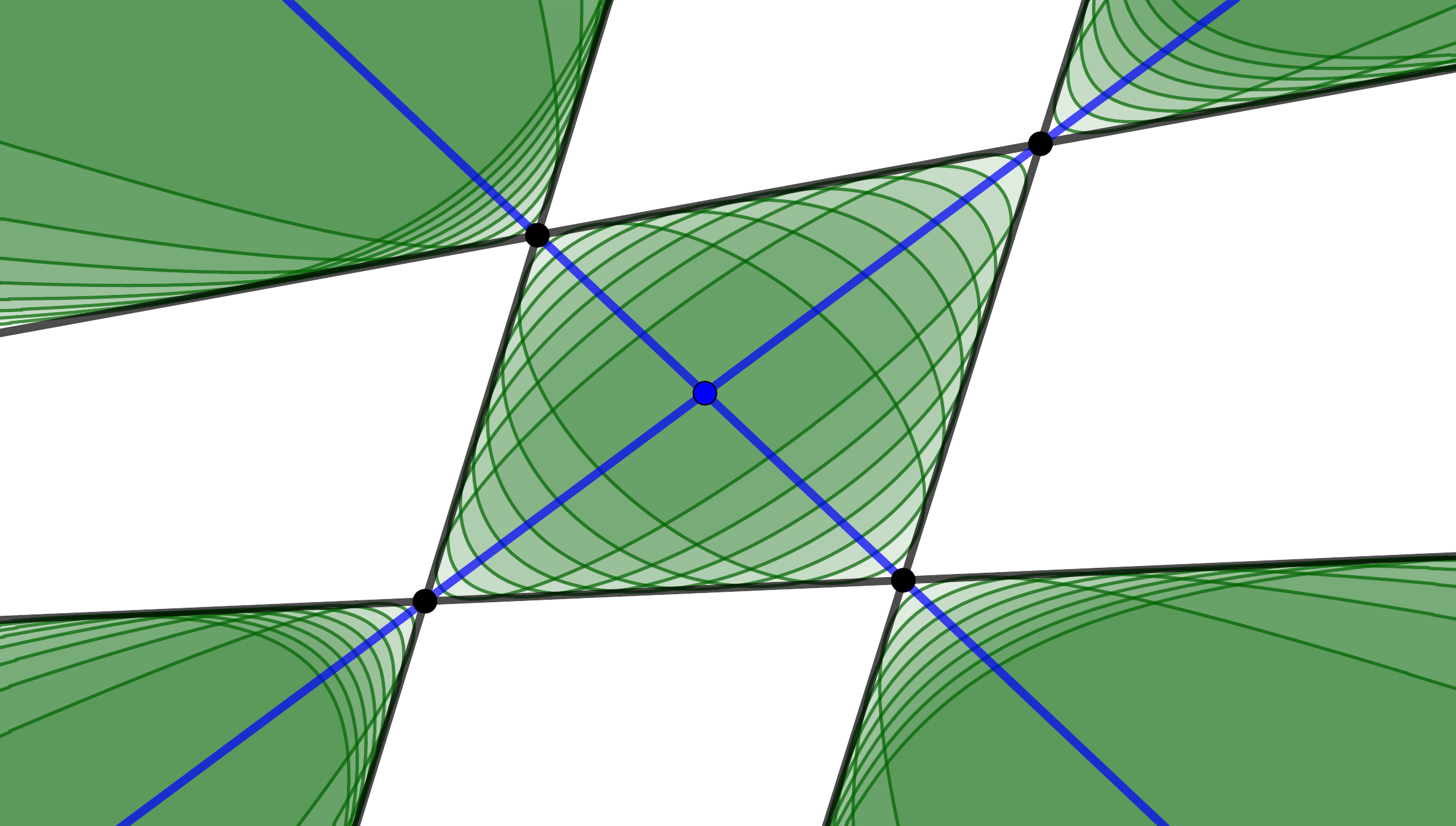}\]
\caption{For any quadrilateral, there is a $1$-parameter family of inscribed conics.}
\label{figure: oneparameterr}
\end{figure}

\begin{cor}\label{cor: 1para}
For any quadrilateral in $\P$, there exists a $1$-parameter family of inscribed conics. Any inscribed conic can be uniquely determined by specifying one of its tangency points which is not a vertex of the quadrilateral.
\end{cor}
\begin{proof}
Consider a quadrilateral $\Box([u], [v], [w], [x])$ in $\P$. Let $[r]$ be the intersection of the diagonals of the quadrilateral. Choose a point of the quadrilateral that is distinct from the vertices of the quadrilateral, say $[\a_1 u+\b_1 v]$ as shown in Figure~\ref{figure: mainlemma}. Construct $[\a_2 v + \b_2 w]$ so that the points $[\a_1 u+\b_1 v]$, $[\a_2 v + \b_2 w]$ and $[r]$ form a Ceva configuration on the triangle $\triangle([u], [v], [w])$. Construct $[\a_3 w+ \b_3 x]$ so that the points $[\a_2 v + \b_2 w]$, $[\a_3 w+ \b_3 x]$ and $[r]$ form a Ceva configuration on the triangle $\triangle ([v], [w], [x])$. Construct $[\a_4 x+ \b_4 u]$ so that the points $[\a_3 w+ \b_3 x]$, $[\a_4 x+ \b_4 u]$ and $[r]$ form a Ceva configuration on the triangle $\triangle ([w], [x], [u])$. By using Ceva's theorem, the incidence theorem in \cite{Meditations} ensures that the points $[\a_4 x+ \b_4 u]$, $[\a_1 u+\b_1 v]$ and $[r]$ form a Ceva configuration on the triangle $\triangle ([x], [u], [v])$. By Theorem~\ref{thm: inquad}, the points $[\a_1 u+\b_1 v]$,  $[\a_2 v + \b_2 w]$, $[\a_3 w+ \b_3 x]$, $[\a_4 x+ \b_4 u]$ are the tangency points of an inscribed conic.
\end{proof}

\begin{figure}[htbp]
\begin{center}
\includegraphics[width =0.8\textwidth]{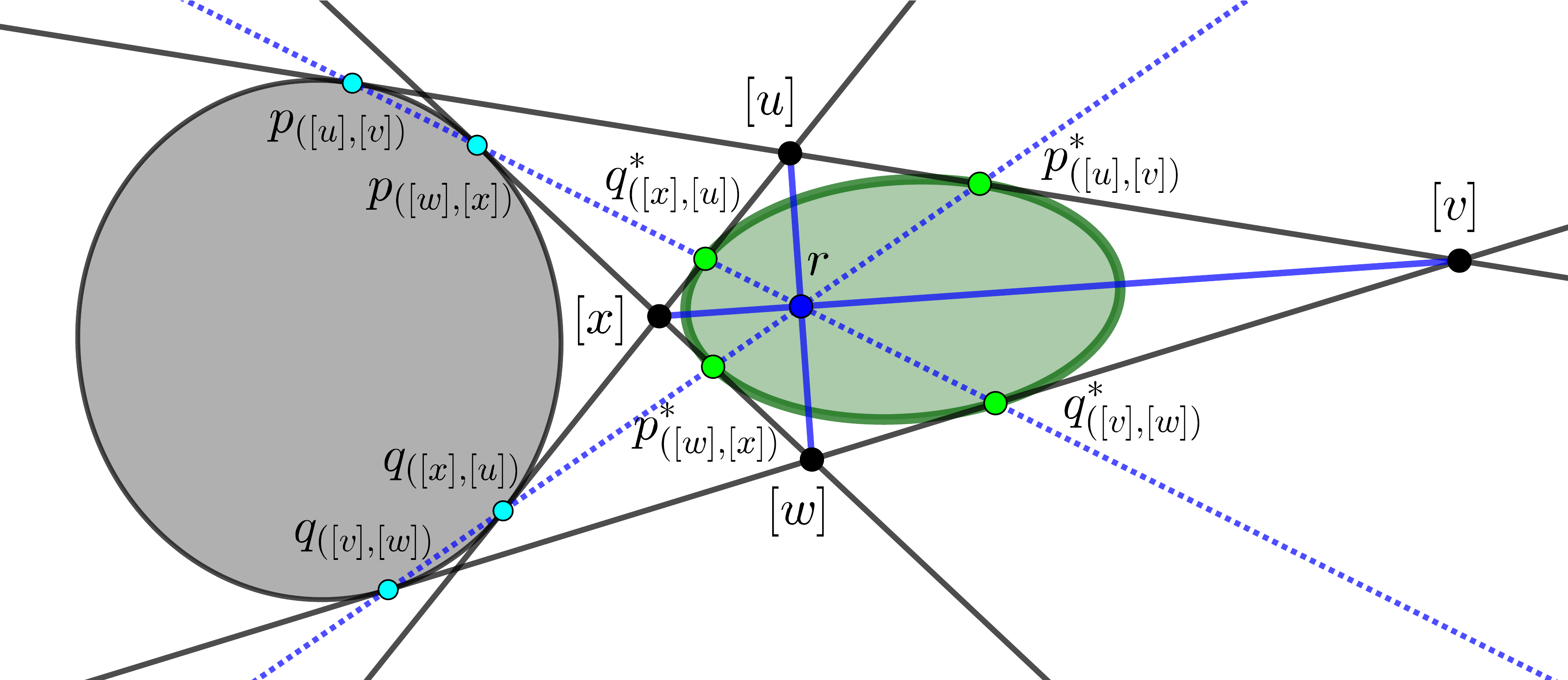}
\end{center}
\caption{The points $p_{([u],[v])}$, $p_{([w],[x])}$, $q_{([x],[u])}$ and $q_{([v],[w])}$ are the tangency points of an inscribed conic. Lemma~\ref{lem: involution} defines an inscribed conic such that the points $p_{([u],[v])}$, $p_{([w],[x])}$, $r$, $q_{([x],[u])}^\ast$, $q_{([v],[w])}^\ast$ are collinear and such that the points $q_{([x],[u])}$, $q_{([v],[w])}$, $r$, $p_{([u],[v])}^\ast$, $p_{([w],[x])}^\ast$ are collinear.}
\label{figure: involution}
\end{figure}

\begin{lem}\label{lem: involution}
Consider a quadrilateral $\Box([u], [v], [w], [x])$ in $\P$ and let $r$ be the intersection point of the diagonals. For any inscribed conic $\mathcal{C}$, let  $p_{([u],[v])}$, $p_{([w],[x])}$, $q_{([x],[u])}$ and $q_{([v],[w])}$ be the tangency points which are labelled by their tangent lines. As shown in Figure~\ref{figure: involution}, draw the two lines containing the collinear points $\{p_{([u],[v])}, r, p_{([w],[x])}\}$ and $\{q_{([x],[u])}, r, q_{([v],[w])})\}$ to construct the points $q_{([v],[w])}^\ast$, $q_{([x],[u])}^\ast$, $p_{([u],[v])}^\ast$, $p_{([w],[x])}^\ast$ which are labelled by their tangent lines. Then, the points $p_{([u],[v])}^\ast$, $p_{([w],[x])}^\ast$, $q_{([v],[w])}^\ast$, $q_{([x],[u])}^\ast$ are the tangency points of an inscribed conic.
\end{lem}

\begin{proof}
By Theorem~\ref{thm: inquad}, the representative vectors $u, v, w, x$ for the vertices of the quadrilateral can be chosen so that $p_{([u],[v])}= [u+v], q_{([v],[w])} = [v+w], p_{([w],[x])} = [w+x], q_{([x],[u])} = [x+u]$ and so that $ r = [u+w]= [v+x]$. Then, $p_{([u],[v])}^\ast = [u-v]$, $p_{([w],[x])}^\ast = [w-x]$, $q_{([x],[u])}^\ast= [x-u]$ and $q_{([w],[v])}^\ast = [w-v]$. The point $[u+w] = [v+x]$ is a point of the four Ceva configurations that are determined by the points $[u-v], [-v+w], [w-x], [-x+u]$ of the triangles $\triangle([u], [v], [w])$, $\triangle ([v], [w], [x])$, $\triangle ([w], [x], [u])$, $\triangle ([x], [u], [v])$. Therefore, by Theorem~\ref{thm: inquad}, the points $[u-v], [-v+w], [w-x], [-x+u]$ are the tangency points of an inscribed conic.
\end{proof}

For any quadrilateral $\Box([u], [v], [w], [x])$ in $\P$, Lemma~\ref{lem: involution} establishes an involution on the $1$-parameter family of inscribed conics. However, there is one degenerate case. For any quadrilateral, there is a unique inscribed conic that is projectively equivalent to a circle inscribed in a square. It is mapped under the involution to a degenerate inscribed conic, namely the double line $(([v],[w])\cap ([x],[u]), ([u], [v]) \cap ([w],[x]))$. We are mostly interested in the generic case.

\section{Nets of planar quadrilaterals with touching inscribed conics}\label{section: nets}

\subsection{Porism}\label{subsection: porism}

In projective space $\Pn$, $n \geq 2$, \emph{nets of planar quadrilaterals} (or Q-nets) are discrete surfaces that are defined by gluing together planar quadrilaterals. By definition, two planar quadrilaterals are glued together if and only if they have two common vertices on a common edge-line. Nets of planar quadrilaterals with touching inscribed conics are nets of planar quadrilaterals such that each planar quadrilateral is equipped with an inscribed conic such that, for any two neighbouring quadrilaterals, the two inscribed conics have the same tangency point on their common edge-line.

A \emph{loop of planar quadrilaterals} is a net of planar quadrilaterals where every quadrilateral is glued with exactly two other quadrilaterals. A loop of planar quadrilaterals is called bipartite if the vertices can be bicoloured so that the vertices have different colours if they share an edge.

\begin{figure}[htbp]
\begin{center}
  \begin{subfigure}[t]{0.4\textwidth}
    \includegraphics[width=\textwidth]{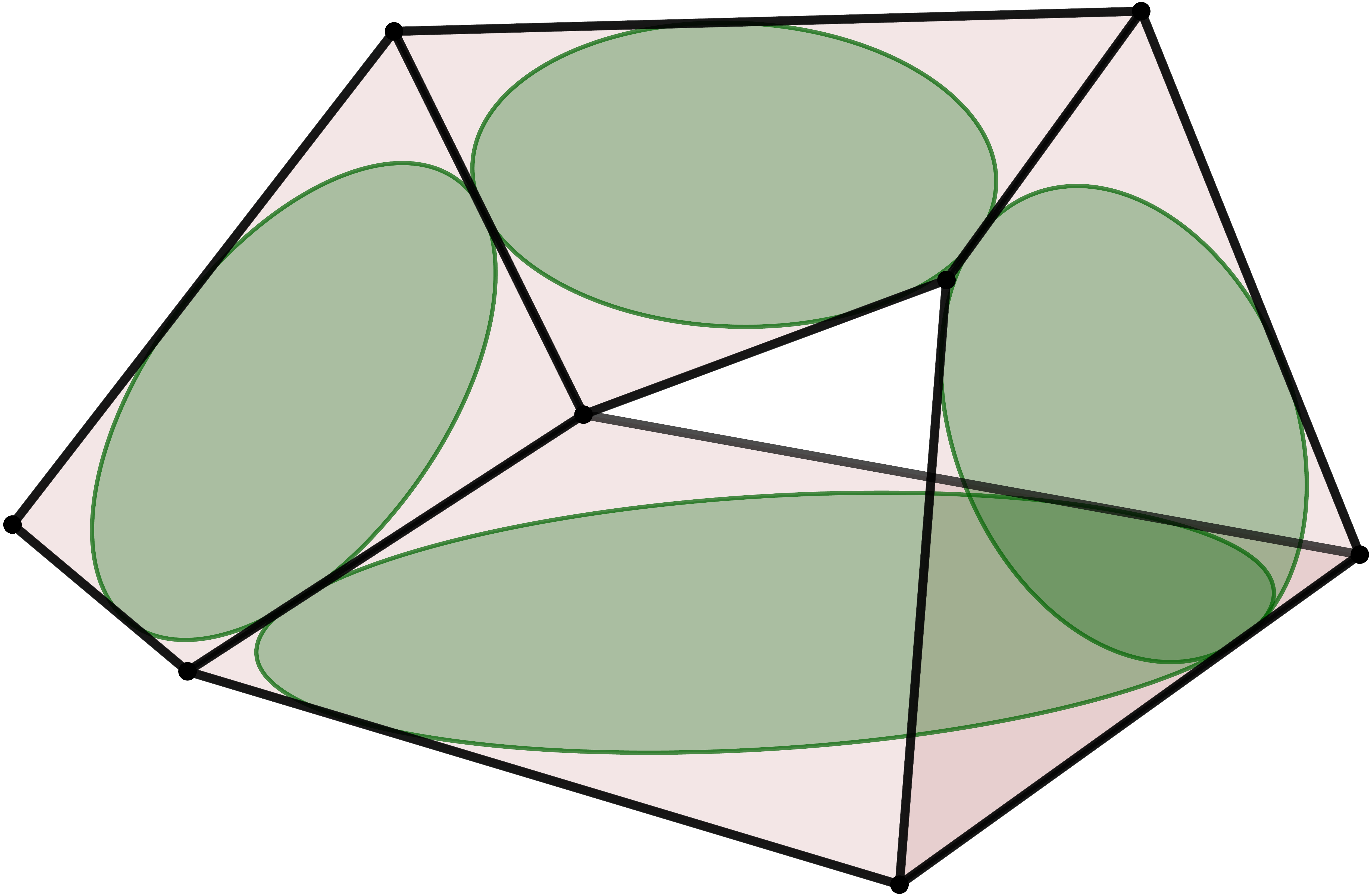}
 \end{subfigure}
  \hspace{0.5cm}%
  \begin{subfigure}[t]{0.4\textwidth}
    \includegraphics[width=\textwidth]{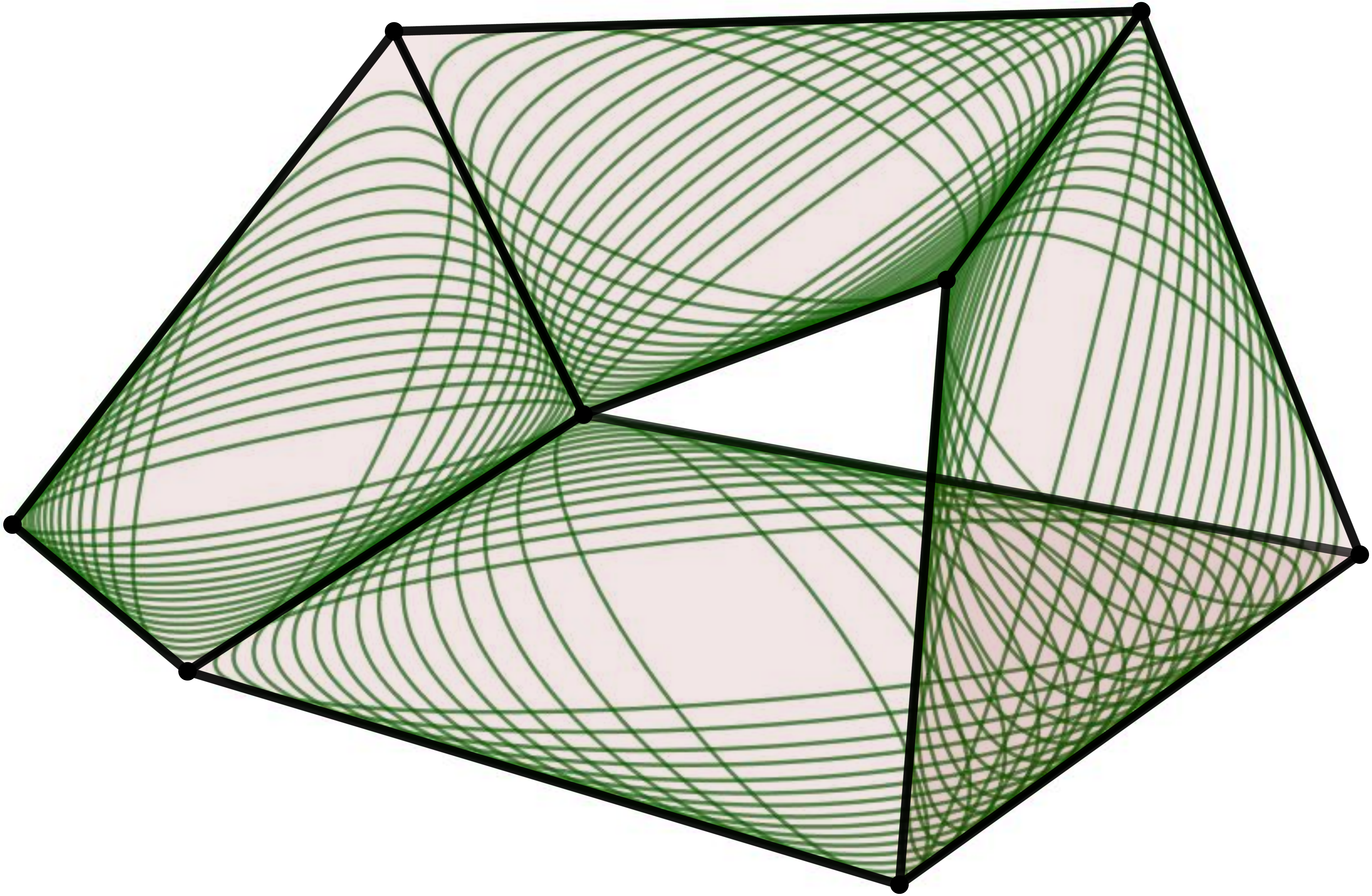}
  \end{subfigure}
  \end{center}
\caption{If a bipartite loop of planar quadrilaterals admits one instance of touching inscribed conics, then it admits a $1$-parameter family of touching inscribed conics.}  \label{figure: mobiusloop}
\end{figure}

\begin{thm}\label{thm: porismloop}
Consider a bipartite loop of finitely many planar quadrilaterals in projective space $\Pn$, $n \geq 2$. If it admits one instance of touching inscribed conics, then it admits a $1$-parameter family of touching inscribed conics.
\end{thm}

\begin{proof}
Enumerate the quadrilaterals $\{Q_i\}_{i=1, \ldots, n}$ of the bipartite loop so that $Q_{i}$ and $Q_{i+1}$ are neighbouring quadrilaterals for any $i \in \Z/n\Z$. Let $l_i$ denote the common edge-line of the two neighbouring quadrilaterals $Q_{i}$ and $Q_{i+1}$. Let $r_i$ denote the intersection of the diagonals of $Q_i$. For each $i \in \Z/n\Z$, define a central projection $f_i: l_{i-1} \to l_{i}$. There are two cases to consider. First, suppose that the lines $l_{i-1}$ and $l_i$ do not intersect at a vertex of the quadrilateral $Q_i$. Then $f_i$ is defined to be the central projection with centre $r_i$. Second, suppose that lines $l_{i-1}$ and $l_i$ do intersect at a vertex of the quadrilateral $Q_i$. Then, the two neighbouring quadrilaterals of $Q_i$ have the incidence structure as shown in Figure~\ref{figure: centralprojection} and the map $f_i : l_{i-1} \to l_i$ is defined so that, for all $q\in l_{i-1}$ distinct from the vertices of $Q_i$, the points $q$, $f_i(q)$ and $r_i$ form a Ceva configuration on the triangle whose vertices are the vertices of $Q_i$ that are also vertices of $Q_{i-1}$ or $Q_{i+1}$. Equivalently, by Proposition~\ref{prop: cevamenelausharmonic}, the map $f_i$ is the central projection whose centre is the point $r_i^\ast$ such that ${\rm{cr}}(p, r_i, f_i(p), r_i^\ast) = -1$ as shown in Figure~\ref{figure: centralprojection}.

\begin{figure}[htbp!]
\begin{center}
\includegraphics[width =0.5\textwidth]{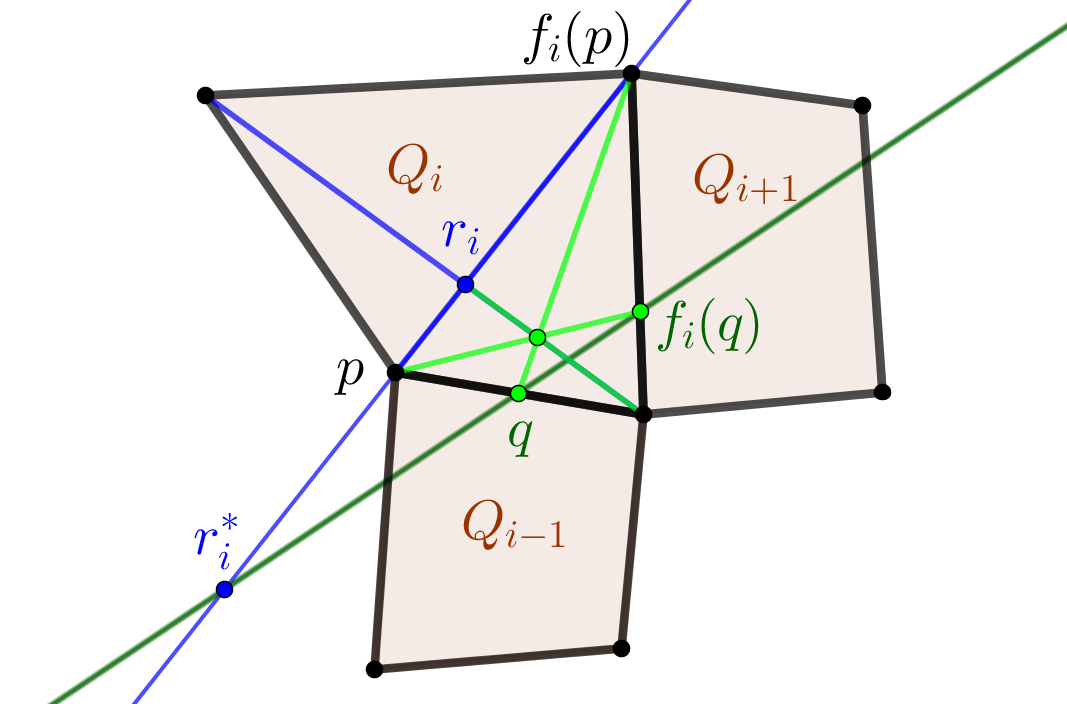}
\end{center}
\caption{The central projection $f_i$ in the case when the quadrilateral $Q_i$ has two neighbouring quadrilaterals sharing a vertex. The centre $r_i^\ast$ of $f_i$ is  defined so that ${\rm{cr}}(p, r_i, f_i(p), r_i^\ast)=-1$. Equivalently, the points $q$, $f_i(q)$ and $r_i$ form a Ceva configuration.} 
\label{figure: centralprojection}
\end{figure}

Because the loop is bipartite, the two common vertices of the quadrilaterals $Q_1$ and $Q_n$ are two fixed points of the projective transformation $f := f_{n} \circ f_{n-1} \circ \ldots \circ f_2 \circ f_1 :l_0 \to l_0$. Suppose that the bipartite loop admits an instance of touching inscribed conics. Then, Theorem~\ref{thm: inquad} implies that the  touching point on the line $l_0$ is a also a fixed point of the projective transformation $f$. Thus, $f \equiv id$. Using Theorem~\ref{thm: inquad}, this shows that the loop admits a $1$-parameter family of touching inscribed conics.
\end{proof}

\begin{figure}[htbp]
\begin{center}
\includegraphics[width =0.6\textwidth]{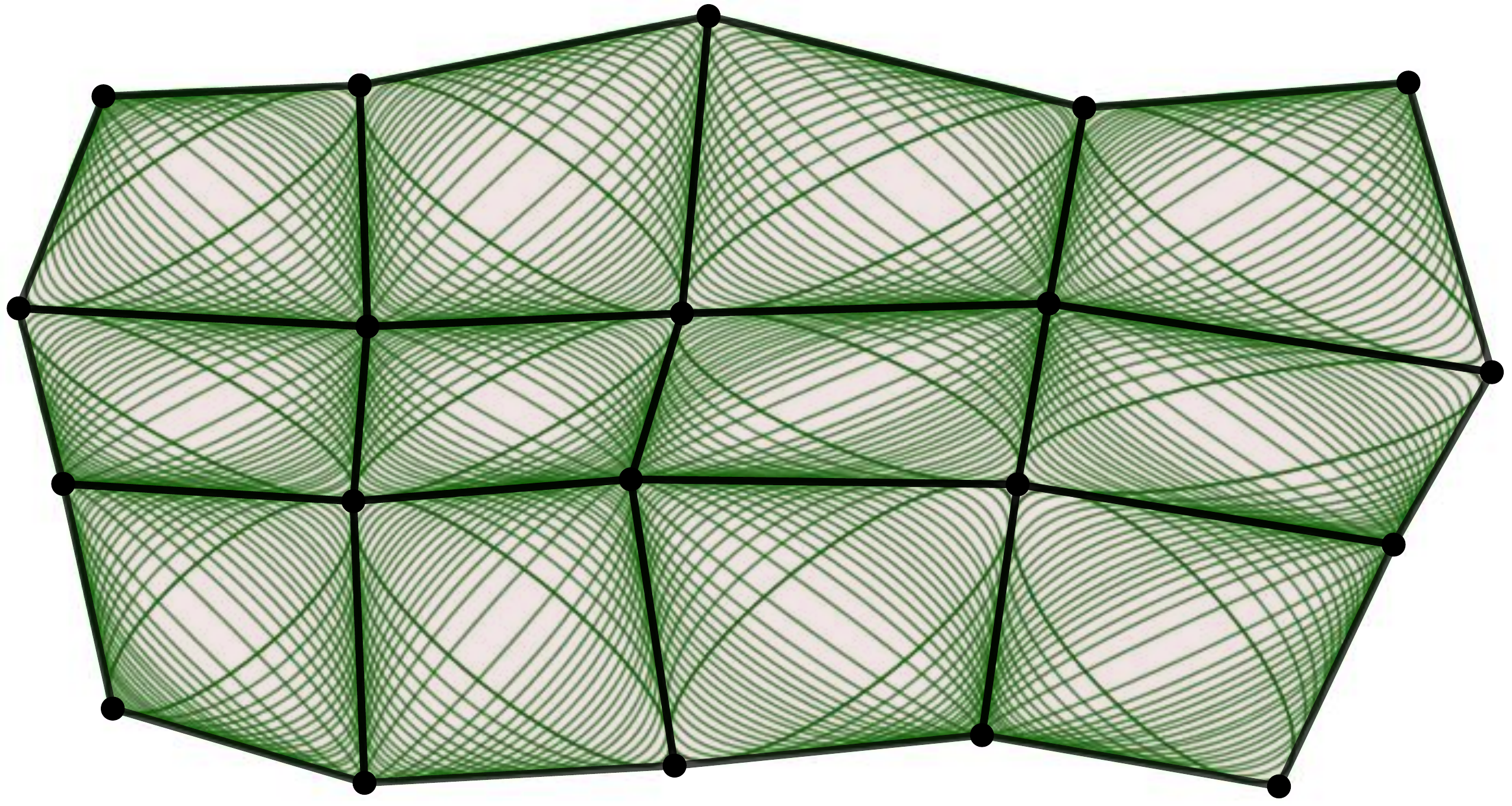}
\end{center}
\caption{If a Q-net admits an instance of touching inscribed conics, then it admits a $1$-parameter family of touching inscribed conics and, by Theorem~\ref{thm: Koenigs}, it is a Koenigs net.}
\label{figure: porism4x3}
\end{figure}

\begin{cor}\label{cor: porismSG}
A Q-net $f: \Z^2 \to \Pn$, $n \geq 2$,  admits an instance of touching inscribed conics if and only if it admits a $1$-parameter family of touching inscribed conics. (See Figure~\ref{figure: porism4x3}.)
\end{cor}

\begin{thm}\label{thm: porismloopnotbipartite}
In projective space $\Pn$, $n \geq 2$, the double cover of a non-bipartite loop always admits a $1$-parameter family of touching inscribed conics.
\end{thm}
\begin{proof}
As in the proof of Theorem~\ref{thm: porismloop}, enumerate the quadrilaterals and let $l_i$ denote the common edge-line of the two neighbouring quadrilaterals $Q_{i}$ and $Q_{i+1}$ for all $i \in \Z/n\Z$. Define $f: l_0 \to l_0$ to be the projective transformation $f_{n} \circ f_{n-1} \circ \ldots \circ f_2 \circ f_1$ that was defined in the proof of Theorem~\ref{thm: porismloop}. Let $v_1$ and $v_2$ be the two common vertices of $Q_1$ and $Q_n$. Because the loop of planar quadrilaterals is not bipartite, it follows that $f(v_1)=v_2$ and $f(v_2)=v_1$. However, any projective transformation $\mathbb{P}^1 \to \mathbb{P}^1$ is an involution if it exhanges two distinct points \cite[Lemma 8.1]{Perspectives}. Therefore, $f \circ f \equiv id$.
\end{proof}

\subsection{Koenigs nets}\label{subsection: koenigsnets}

Two planar quadrilaterals $\Box(A,B,C,D)$ and $\Box(A^*,B^*,C^*, D^*)$ are called \emph{dual quadrilaterals} if and only if their corresponding edge-lines are parallel and their non-corresponding diagonals are parallel. For any planar quadrilateral, a dual quadrilateral exists and it is uniquely determined up to translation and rescaling.

A net $f: \Z^2 \to \mathbb{A}^n$ of planar quadrilaterals in affine space $\mathbb{A}^n$, $n \geq 2$,  is called a \emph{$2$-dimensional Koenigs} if and only if there exists a Christoffel dual net $f^\ast: \Z^2 \to \mathbb{A}^n$ such that the corresponding quadrilaterals are dual \cite{DDG}. Although $2$-dimensional Koenigs nets are defined in terms of affine geometry, it is known that the class of $2$-dimensional Koenigs nets is invariant under projective transformations.

\begin{defn}\label{defn: koenigsexact}
Consider a Q-net $f: \Z^2 \to \mathbb{A}^n$, $n \geq 2$. Denote by $M_{i,j}$ the intersection point of the diagonals of the quadrilateral $\Box(f_{i,j}, f_{i+1,j}, f_{i+1,j+1}, f_{i, j+1})$. Then, the net $f: \Z^2 \to \mathbb{A}^n$ is a $2$-dimensional Koenigs net if and only if the following condition is satisfied for all $(i,j) \in \Z^2$.
\begin{align}\label{eqn: koenigsdefn}
\frac{l(M_{i,j}, f_{i+1,j})}{l(M_{i,j},f_{i,j+1})}\cdot\frac{l(M_{i-1,j}, f_{i,j+1})}{l(M_{i-1,j},f_{i-1,j})}\cdot\frac{l(M_{i-1,j-1},f_{i-1,j})}{l(M_{i-1,j-1},f_{i,j-1})}\cdot\frac{l(M_{i,j-1},f_{i,j-1})}{l(M_{i,j-1},f_{i+1,j})} = 1
\end{align} 
\end{defn}

This algebraic characterisation and further projective geometric properties of Koenigs nets can be found in \cite{DDG}.

\begin{figure}[htbp]
\begin{center}
\includegraphics[width =0.5\textwidth]{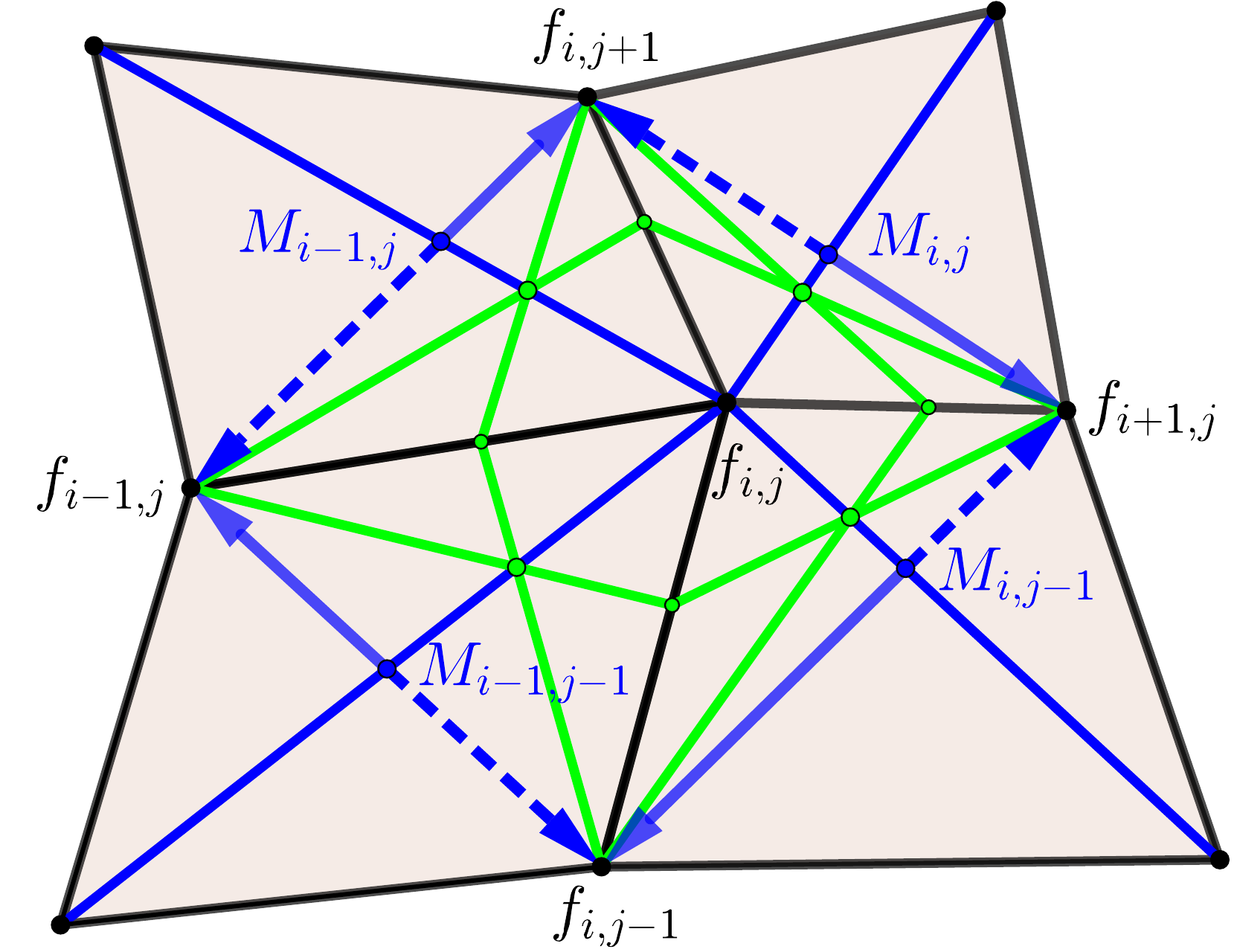}
\end{center}
\caption{The identity (\ref{eqn: koenigsdefn}) is valid if and only if the net of planar quadrilaterals admits a $1$-parameter family of touching inscribed conics.}
\label{figure: Koenigsdegree4}
\end{figure}

\begin{thm}\label{thm: Koenigs}
A Q-net $f: \Z^2 \to \mathbb{P}^n$, $n \geq 2$,  is a $2$-dimensional Koenigs net if and only if it admits (a $1$-parameter family of) touching inscribed conics.
\end{thm}

\begin{proof}
Ceva's theorem implies that equation (\ref{eqn: koenigsdefn}) from Definition~\ref{defn: koenigsexact} is a necessary and sufficient condition for the existence of a $1$-parameter family of Ceva configurations that are glued together as shown in Figure~\ref{figure: Koenigsdegree4}. Equivalently, by Theorem~\ref{thm: inquad}, there is a $1$-parameter family of touching inscribed conics.
\end{proof}

\begin{cor}\label{cor: koenigsquadric}
If all the edge-lines of a Q-net $f: \Z^2 \to \mathbb{P}^n$, $n \geq 2$, are tangent to a non-degenerate quadric, then $f$ is a $2$-dimensional Koenigs net.
\end{cor}

Examples of Q-nets with their edge-lines tangent to a sphere are given by Koebe polyhedra, which are used in \cite{BHS06} to construct discrete minimal surfaces. The corresponding touching conics are circles. Koebe polyhedra have a $1$-parameter family of touching inscribed conics. 
\section{Line grids with quadrilaterals with touching inscribed conics}\label{section: grids}

\subsection{Polygonal chains inscribed in conics}\label{subsection: projectivebilliards}

Let $p_0, p_1, \ldots, p_m$ be the vertices of a polygonal chain that is inscribed in a non-degenerate conic $\mathcal{C}$ and let $q_0, q_1, \ldots, q_n$ be the vertices of another polygonal chain that is also inscribed in the non-degenerate conic $\mathcal{C}$. Let $k_0, k_1, \ldots, k_m$ and $l_0, l_1, \ldots, l_n$ be the tangent lines of $\mathcal{C}$ at the points $p_0, p_1, \ldots, p_m$ and $q_0, q_1, \ldots, q_n$. For any $i,j \in \N$ such that $1 \leq i  \leq m$ and $1 \leq j \leq n$, the notation $\Box (k_{i-1},l_{j-1}, k_i, l_j)$ denotes the quadrilateral with the vertices $k_{i-1}\cap l_{j-1}, k_{i-1}\cap l_{j}, k_{i}\cap l_{j}, k_{i}\cap l_{j-1}$. Define the two lines $k_{i-1,i} := (p_{i-1} , p_i)$ and $l_{j,j-1} := (q_{j-i} , q_j)$. By Lemma~\ref{lem: involution}, the points $k_{i-1,i} \cap l_{j-1}$, $k_{i-1,i} \cap l_j$, $k_{i-1} \cap l_{j-1,j}$, $k_{i} \cap l_{j-1,j}$ are the tangency points of a conic that is inscribed in the quadrilateral $\Box (k_{i-1},l_{j-1}, k_i, l_j)$. Therefore, the $m\times n$ grid of quadrilaterals  $\{\Box (k_{i-1},l_{j-1}, k_i, l_j)\}_{1 \leq i \leq m, 1 \leq j \leq n}$ admits an instance of touching inscribed conics such that the tangency points satisfy some non-trivial collinearites. An example is shown in Figure~\ref{figure: TCnethyp2polygonalchains} where the non-trivial collinearities are represented by the dotted lines.
 
\begin{figure}[htbp]
\begin{center}
\includegraphics[width =0.9\textwidth]{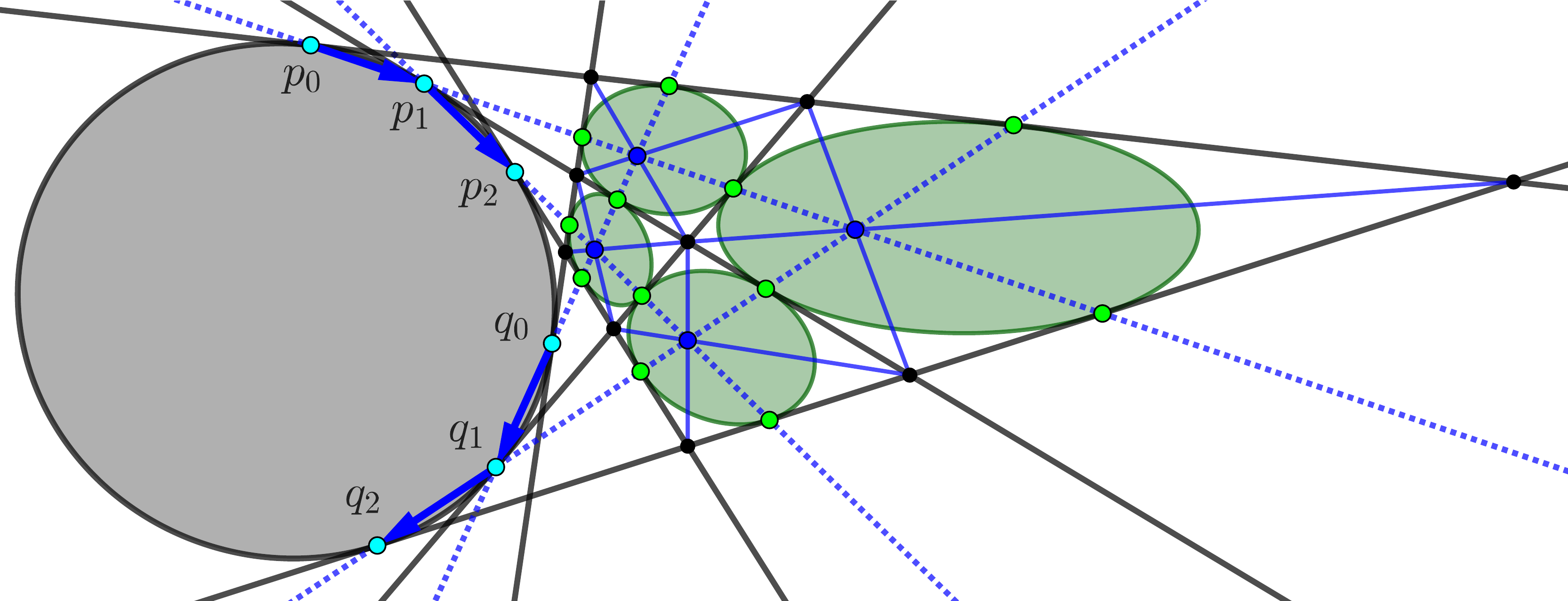}
\end{center}
\caption{Two polygonal chains $p_0, p_1, p_2$ and $q_0, q_1, q_2$ are inscribed in a conic. The solid lines are the tangent lines $k_0, k_1, k_2$ and $l_0, l_1, l_2$. The dotted lines are the lines $(p_0, p_1)$, $(p_1, p_2)$, $(q_0, q_1)$, $(q_1, q_2)$. By construction, the tangency points of the touching inscribed conics lie on the dotted lines.}
\label{figure: TCnethyp2polygonalchains}
\end{figure} 

In the above construction of grids of quadrilaterals with touching inscribed conics, the two polygonal chains determine the ``horizontal" and ``vertical" lines of the grids. However, they can be merged. (See Figure~\ref{figure: TCperiodictrajectorysamesize}).

\begin{figure}[htbp]
\begin{center}
\includegraphics[width =0.75\textwidth]{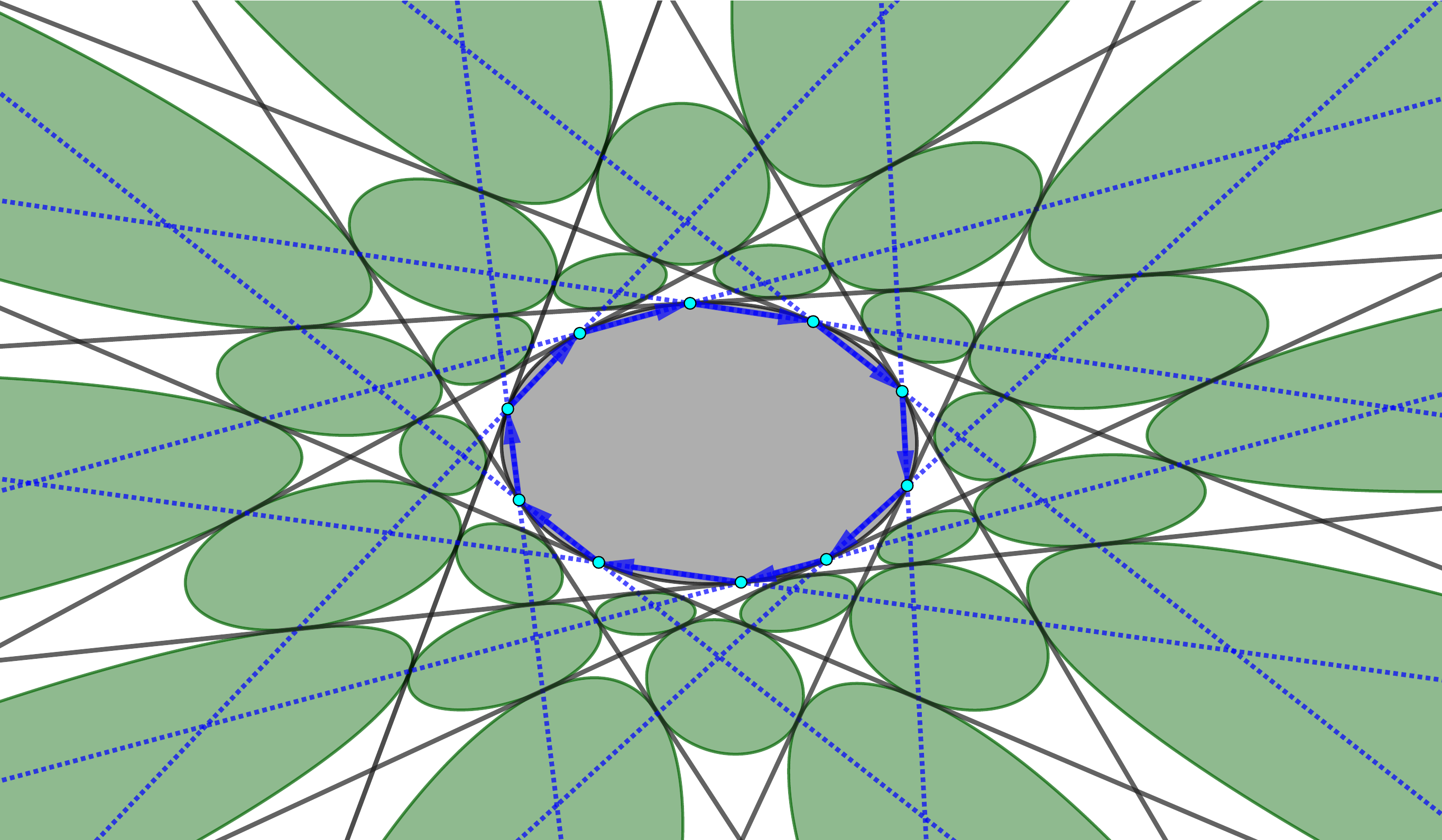}
\end{center}
\caption{An inscribed polygon and a grid of quadrilaterals with touching inscribed conics. Among the $1$-parameter family of touching inscribed conics, there is one instance of touching inscribed conics such that the tangency points of the touching inscribed conics are contained in the dotted edge-lines of the inscribed polygon.}
\label{figure: TCperiodictrajectorysamesize}
\end{figure}

\subsection{Generic lines tangent to a conic}\label{subsection: sixlinestangenttoaconic}

Theorem~\ref{thm: tangentlines} is a consequence of two classical theorems which are referenced in the proof.

\begin{thm}\label{thm: tangentlines}
Let $k_0, k_1$ and $l_0, l_1, \ldots, l_n$, $n \geq 2$, be lines in the projective plane such that each quadrilateral $Q_i := \Box (k_0, l_{i-1}, k_1, l_i)$ has four generic edge-lines. Let $r_i$ be the intersection point of the diagonals of $Q_i$. Then, the following are equivalent.
\begin{enumerate}[label=(\roman*)]
\item The lines $k_0, k_1$ and $l_0, l_1, \ldots,  l_n$ are tangent to a non-degenerate conic.
\item The points  $\{r_{i}\}_{1 \leq i \leq n}$ lie in a line that does not contain the point $k_0 \cap k_1$.
\end{enumerate}
Suppose that the lines $k_0, k_1$ and $l_0, l_1, \ldots,  l_n$ are tangent to a non-degenerate conic $\mathcal{C}$. Let $k_{0,1}$ be the line containing the collinear points $\{r_{i}\}_{1 \leq i \leq n}$. Then, $k_0 \cap k_{0,1}$ and $k_1 \cap k_{0,1}$ are the tangency points of the tangent lines $k_0$ and $k_1$.
\end{thm}

\begin{proof}
By the dual of Steiner's theorem on the projective generation of non-degenerate conics, the lines $k_0, k_1, l_0, l_1, \ldots,  l_n$ are tangent to a non-degenerate conic if and only if there is a projective transformation $f: k_0 \to k_1$ such that $f(k_0\cap l_i) = k_1\cap l_i$ for all $i \in \{1, \ldots, n\}$ and such that $f(k_0 \cap k_1) \neq k_0 \cap k_1$ \cite[Theorems 8.1.4 and 8.1.8]{Casas}. Equivalently, by the cross-axis theorem \cite[Theorem 5.3.5]{Casas} and \cite[Proposition 5.3.7]{Casas}, the points $\{r_i\}_{1 \leq i \leq n}$ are contained in a line which is called the cross-axis of $f: k_0 \to k_1$. The cross-axis is not concurrent with the lines $k_0$ and $k_1$ because otherwise the projective transformation $f: k_0 \to k_1$ would be a central projection so that $f(k_0 \cap k_1) = k_0 \cap k_1$. Therefore $(i)$ and $(ii)$ are equivalent. 

Suppose that the generic lines $k_0, k_1$ and $l_0, l_1, \ldots,  l_n$ are tangent to a non-degenerate conic $\mathcal{C}$. Let $p_0$ and $p_1$ be the tangency points of the the tangent lines $k_0$ and $k_1$. Because $\mathcal{C}$ is inscribed in each of the quadrilaterals $\{Q_i\}_{i=1, \ldots, n}$, Proposition~\ref{prop: opptangencypts} ensures that the points $\{r_i\}_{i=1, \ldots, n}$ are contained in the line $(p_0, p_1)$. Therefore, $p_0 = k_0 \cap k_{0,1}$ and $p_{1} = k_1 \cap k_{0,1}$.
\end{proof}

Let $k_0, k_1, \ldots, k_m$ and $l_0, l_1, \ldots, l_n$ be generic lines in the projective plane. Consider the $m\times n$ grid of quadrilaterals $Q_{i,j} := \Box (k_{i-1}, l_{j-1}, k_i, l_j)$. We use $K_{i-1,i}$ and $L_{j-1,j}$ to denote the strips of quadrilaterals $\{\Box(k_{i-1}, l_{j-1}, k_i, l_j)\}_{j=1, \ldots, n}$ and $\{\Box(k_{i-1},l_{j-1}, k_{i}, l_j)\}_{i=1, \ldots, m}$, respectively.

\begin{thm}\label{thm: mainthm}
For six generic lines $k_0, k_1, k_2, l_0, l_1, l_2$ in the projective plane, consider the $2\times 2$ grid of quadrilaterals $Q_{i,j} := \Box(k_{i-1},l_{j-1},k_i,l_j)$. We use $r_{i,j}$ to denote the intersection point of the diagonals of the quadrilateral $Q_{i,j}$. Then, the following are equivalent.
\begin{enumerate}[label=(\roman*)]
\item The six lines $k_0, k_1, k_2, l_0, l_1, l_1$ are tangent to a non-degenerate conic.
\item The $2 \times 2$ grid of quadrilaterals admits an instance of touching inscribed conics $C_{i,j}$ such that the following sets are sets of collinear points. (See Figure~\ref{figure: maintheoremproof}.)
 \begin{align*}
\{K_{0,1}l_0, r_{1,1}, K_{0,1}l_{1}, r_{1,2}, K_{0,1}l_{2}\} \qquad \{k_{0}L_{1,2}, r_{1,2}, k_{1}L_{1,2}, r_{2,2}, k_{2}L_{1,2}\}\\
\{K_{1,2}l_{0}, r_{2,1}, K_{1,2}l_{1}, r_{2,2}, K_{1,2}l_{2}\} \qquad \{k_{0}L_{0,1}, r_{1,1}, k_{1}L_{0,1},r_{2,1}, k_{2}L_{0,1}\}
\end{align*}
The points $K_{i-1,i}l_{j-1}$, $K_{i-1,i}l_j$, $k_{i-1}L_{j-1,j}$, $k_iL_{j-1,j}$ are defined to be the tangency points of the conic $\mathcal{C}_{i,j}$ that is inscribed in the quadrilateral $Q_{i,j}$. The tangency points are labelled by their tangent lines and by the strips of quadrilaterals. 
\item The $2\times 2$ grid of quadrilaterals admits an instance of touching inscribed conics.
\item The $2\times 2$ grid of quadrilaterals admits a $1$-parameter family of touching inscribed conics.
\item The three lines $(r_{1,1}, r_{2,1})$,  $(r_{1,2}, r_{2,2})$ and $l_1$ are concurrent.
\item The three lines $(r_{1,1} , r_{1,2})$, $(r_{2,1}, r_{2,2})$ and $k_1$ are concurrent.
\end{enumerate}
\end{thm}

\begin{figure}[t]
\begin{center}
\includegraphics[width =0.9\textwidth]{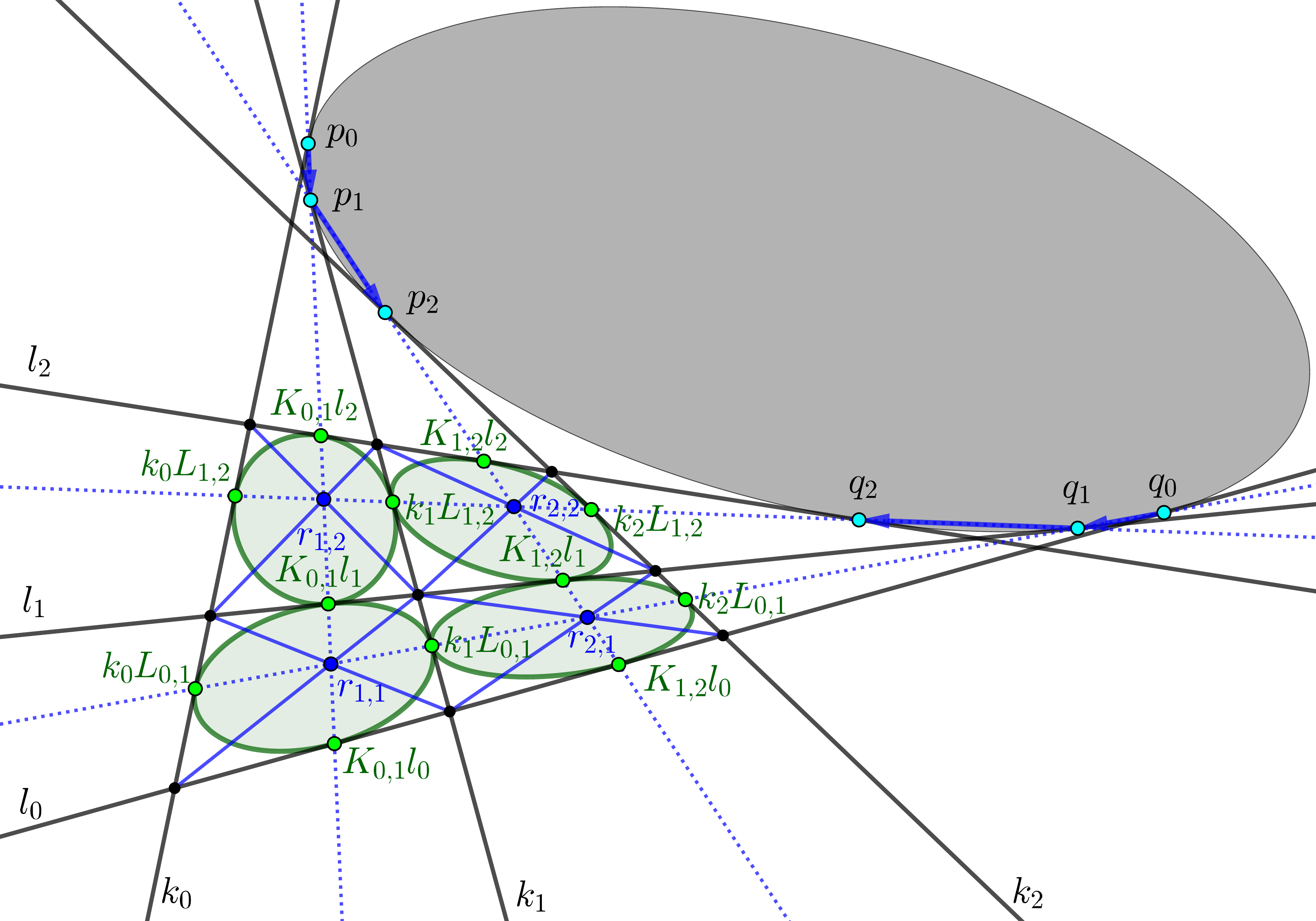}
\end{center}
\caption{The six generic gridlines are tangent to a non-degenerate conic if and only if the $2\times 2$ grid of quadrilaterals admits a $1$-parameter family of touching inscribed conics. Among the $1$-parameter family of touching inscribed conics, there is one instance of touching inscribed conics such that the tangency points satisfy some non-trivial collinearities which are represented by the dotted lines.}
\label{figure: maintheoremproof}
\end{figure}

\begin{proof}
Suppose that the generic lines $k_0, k_1, k_2, l_0, l _1, l_2$ are tangent to a non-degenerate conic $\mathcal{C}$. Let $p_0, p_1, p_2, q_0, q_1, q_2$ be the tangency points of the tangent lines $k_0, k_1, k_2, l_0, l _1, l_2$. Consider the two polygonal chains $p_0, p_1, p_2$ and $q_0, q_1, q_2$ that are inscribed in the non-degenerate conic $\mathcal{C}$. By the construction in Section~\ref{subsection: projectivebilliards}, the $2\times 2$ grid of quadrilaterals admits an instance of touching inscribed conics such that the following sets are sets of collinear points.
\begin{align*}
\{K_{0,1}l_0, r_{1,1}, K_{0,1}l_{1}, r_{1,2}, K_{0,1}l_{2}\} \qquad \{k_{0}L_{1,2}, r_{1,2}, k_{1}L_{1,2}, r_{2,2}, k_{2}L_{1,2}\}\\
\{K_{1,2}l_{0}, r_{2,1}, K_{1,2}l_{1}, r_{2,2}, K_{1,2}l_{2}\} \qquad \{k_{0}L_{0,1}, r_{1,1}, k_{1}L_{0,1},r_{2,1}, k_{2}L_{0,1}\}
\end{align*}
Therefore, \textit{(i)} implies \textit{(ii)}. Obviously, $(ii)$ implies $(iii)$. By Theorem~\ref{thm: porismloop}, \textit{(iii)} implies \textit{(iv)}. 

Suppose that the $2 \times 2$ grid of quadrilaterals admits a $1$-parameter family of touching inscribed conics. By Theorem~\ref{thm: Koenigs}, it is a Koenigs net. So, in any affine image of $\P$, 
\begin{align}\label{eq: mr1}
\frac{l(k_0\cap l_1,r_{1,1})}{l(r_{1,1},k_1\cap l_0)}\frac{l(k_1\cap l_0,r_{2,1})}{l(r_{2,1},k_2\cap l_1)}\frac{l(k_2\cap l_1,r_{2,2})}{l(r_{2,2},k_1\cap l_2)}\frac{l(k_1\cap l_2,r_{1,2})}{l(r_{1,2},k_0\cap l_1)}=1.
\end{align}
By applying Menelaus's theorem to the triangles $\triangle(k_0\cap l_1,k_2\cap l_1,k_1\cap l_2)$ and $\triangle(k_0\cap l_1, k_2\cap l_1, k_1\cap l_0)$, the identity~(\ref{eq: mr1}) implies that the two lines $(r_{1,1},r_{2,1})$ and $(r_{1,2}, r_{2,2})$ are concurrent with the line $l_1$. By applying Menelaus' theorem to the triangles $\triangle(k_1\cap l_0,k_1\cap l_2,k_0\cap l_1)$ and $\triangle(k_1\cap l_0, k_1\cap l_2, k_2\cap l_1)$, the identity~(\ref{eq: mr1}) implies that the two lines $(r_{1,1},r_{1,2})$ and $(r_{2,1}, r_{2,2})$ are concurrent with the line $k_1$. Therefore, $(iv)$ implies both $(v)$ and $(vi)$.

Suppose that the three lines $(r_{1,1}, r_{2,1})$,  $(r_{1,2}, r_{2,2})$ and $l_1$ are concurrent. Let $q_1$ be the concurrency point. The generic lines $k_0, k_1, k_2, l_0, l_1$ are tangent to a uniquely determined non-degenerate conic \cite[Corollary 8.1.12]{Casas}, say $\mathcal{A}$ . By Theorem~\ref{thm: tangentlines}, $q_1$ is a tangency point of $\mathcal{A}$. Likewise, the generic lines $k_0, k_1, k_2, l_1, l_2 $ are tangent to a uniquely determined non-degenerate conic, say $\mathcal{B}$, with the tangency point $q_1$. Then, by Corollary~\ref{cor: 1para}, $\mathcal{A} = \mathcal{B}$ because $\mathcal{A}$ and $\mathcal{B}$ have four common generic tangent lines $k_0, k_1, k_2, l_1$ and the common tangency point $q_1$. Therefore, $(v)$ implies $(i)$. Symmetrically, $(vi)$ also implies $(i)$. 
\end{proof}

\begin{cor}\label{cor: tangentgridlines}
Let $k_0, k_1, \ldots, k_m$, $m \geq 2$ and $l_0, l_1, \ldots, l_n$, $m \geq 3$, be generic lines in the projective plane. Consider the $m\times n$ grid of quadrilaterals $Q_{i,j} := \Box (k_{i-1}, l_{j-1}, k_i, l_j)$. We use $r_{i,j}$ to denote the intersection point of the diagonals of the quadrilateral $Q_{i,j}$. Then, the following are equivalent.
\begin{enumerate}[label=(\roman*)]
\item The generic lines  $k_0, k_1, \ldots, k_m, l_0, l_1, \ldots, l_n$ are tangent to a non-degenerate conic.
\item The $m \times n$ grid of quadrilaterals admits an instance of touching inscribed conics $C_{i,j}$ such that the following are collections of sets of collinear points. (See Figure~\ref{figure: fsafbabf}.) \begin{align*}
\{\{K_{i-1,i}l_j\}_{j= 0, \ldots, n}\}_{i = 1, \ldots, m} \qquad \{\{k_{i}L_{j-1,j}\}_{i= 0, \ldots, m}\}_{j = 1, \ldots, n}
\end{align*}
The points $K_{i-1,i}l_{j-1}$, $K_{i-1,i}l_j$, $k_{i-1}L_{j-1,j}$, $k_iL_{j-1,j}$ are defined to be the tangency points of the conic $\mathcal{C}_{i,j}$ that is inscribed in the quadrilateral $Q_{i,j}$. The tangency points are labelled by their tangent lines and by the strips of quadrilaterals.
\item The $m\times n$ grid of quadrilaterals admits an instance of touching inscribed conics.
\item The $m\times n$ grid of quadrilaterals admits a $1$-parameter family of touching inscribed conics.
\item $\{ \{r_{i,j}\}_{i=1, \ldots, m}\}_{j=1, \ldots, n}$ and $\{ \{r_{i,j}\}_{j=1, \ldots, n}\}_{i=1, \ldots, m}$ are collections of sets of collinear points.
\item $\{ \{r_{i,j}\}_{j=1, \ldots, n}\}_{i=1, \ldots, m}$ is a collection of sets of collinear points.
\end{enumerate}
\end{cor}

\begin{proof}

Analogously to the proof of Theorem~\ref{thm: mainthm}, the implications $(i) \implies (ii) \implies (iii) \implies (iv)\implies (v)\implies (vi)$ are straightforward. The only step we comment is $(vi) \implies (i)$. Suppose that $\{ \{r_{i,j}\}_{j=1, \ldots, n}\}_{i=1, \ldots, m}$ is a collection of sets of collinear points. By Theorem~\ref{thm: tangentlines}, for any $i  \in \{1, \ldots, m\}$, the generic lines $k_{i-1}, k_i, l_0, l_1, \ldots, l_n$ are tangent to a non-degenerate conic, say $\mathcal{C}_i$. For any $i  \in \{1, \ldots, m-1\}$, the non-degenerate conics $\mathcal{C}_i$ and $\mathcal{C}_{i+1}$ are identical because they have five common tangent lines $k_i, l_0, l_1, l_2, l_3$. Therefore, $(vi) \implies (i)$. 
\end{proof}

\begin{figure}[htbp]
\begin{center}
  \begin{subfigure}[t]{0.45\textwidth}
   \includegraphics[width=\textwidth]{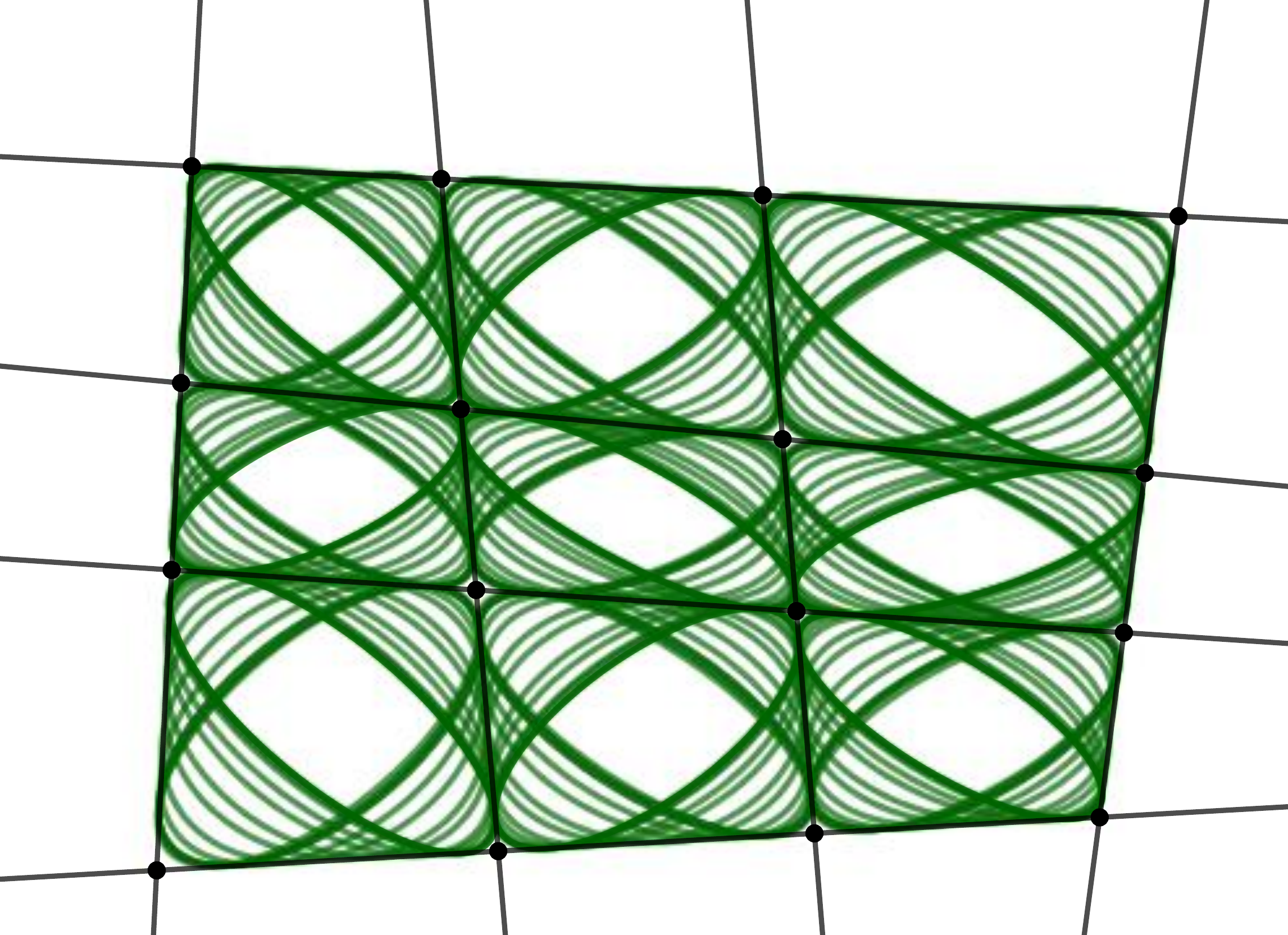}
 \label{figure: fdbdfb}
 \end{subfigure}
\hspace{0.5cm}%
  \begin{subfigure}[t]{0.45\textwidth}
    \includegraphics[width=\textwidth]{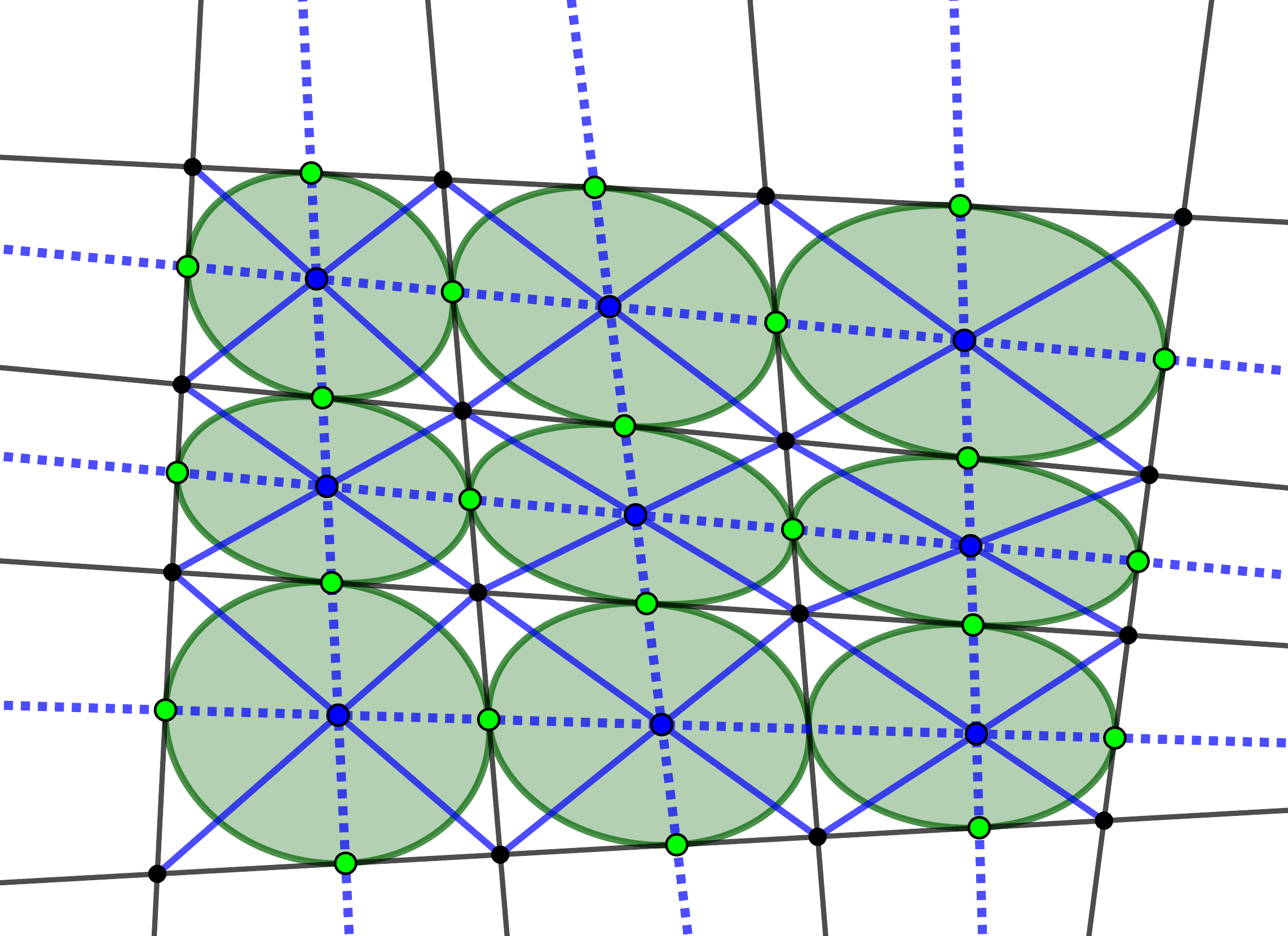}
\label{figure: fdbadfbfbd}
  \end{subfigure}
  \end{center}
  \caption{If a $3\times 3$ grid of quadrilaterals admits an instance of touching inscribed conics, then there is a $1$-parameter family of touching inscribed conics. Among the $1$-parameter family, there is one instance of touching inscribed conics such that the tangency points satisfy some non-trivial collinearities which are represented by the dotted lines.}
  \label{figure: fsafbabf}
\end{figure}

\begin{figure}[htbp]
\begin{center}
    \includegraphics[width=0.65\textwidth]{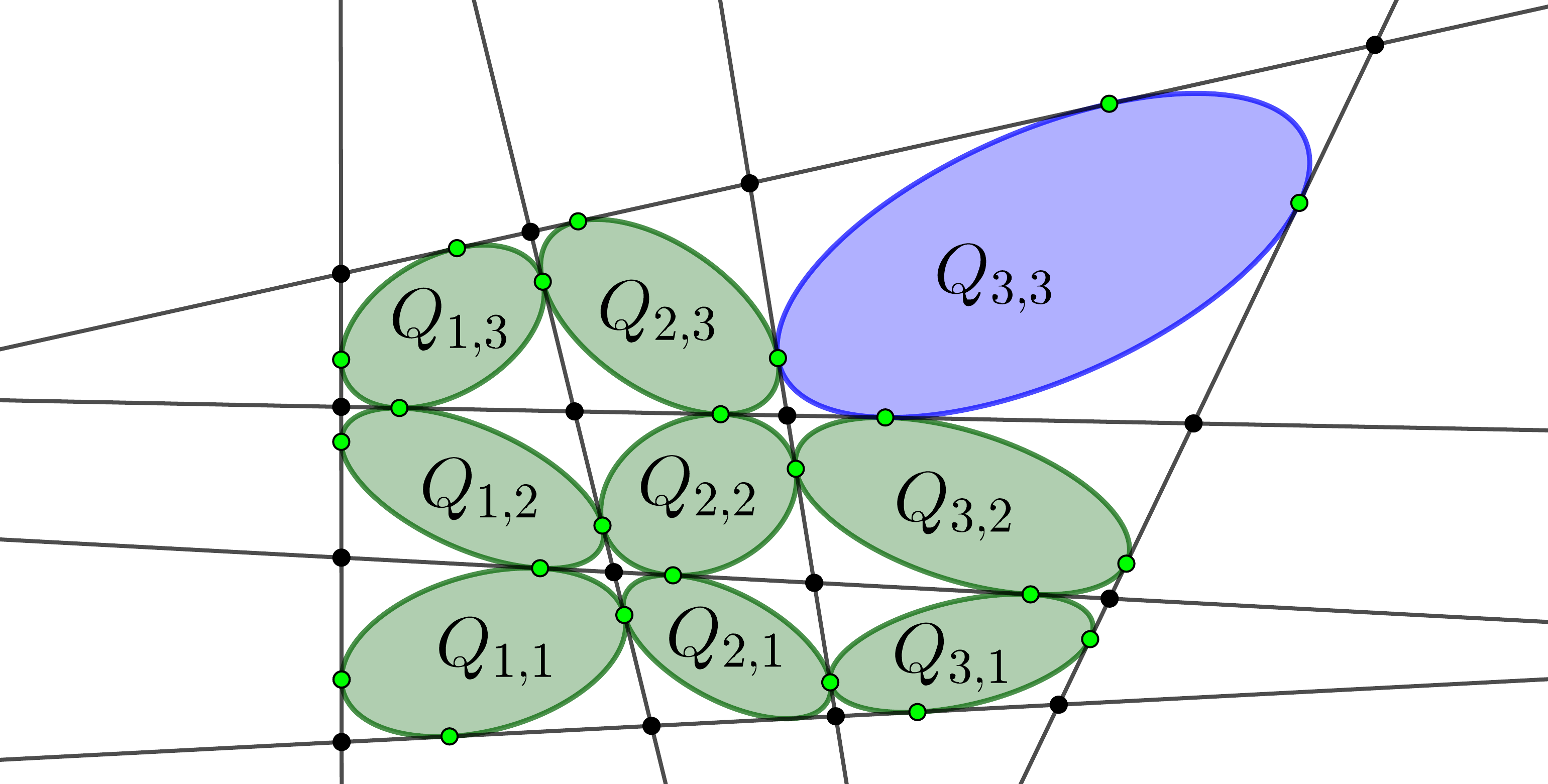}
\end{center}
\caption{A $3 \times 3$ grid of quadrilaterals $Q_{i,j}$. Suppose that each of the quadrilaterals $Q_{1,1}$, $Q_{1,2}$, $Q_{1,3}$, $Q_{2,1}$, $Q_{2,2}$, $Q_{2,3}$, $Q_{3,1}$, $Q_{3,2}$ is equipped with an inscribed conic such that, for any two neighbouring quadrilaterals, the inscribed conics are touching. Then, the quadrilateral $Q_{3,3}$ admits an inscribed conic that touches the two conics that are inscribed in the quadrilaterals $Q_{3,2}$ and $Q_{2,3}$. By Corollary~\ref{cor: tangentgridlines}, the eight lines are tangent to a conic.}
       \label{figure: 3times3proof}
\end{figure}

Koenigs nets can be treated as discrete conjugate nets with equal Laplace invariants \cite{DDG}. By Theorem~\ref{thm: Koenigs}, grids of quadrilaterals with touching inscribed conics are planar $2$-dimensional Koenigs nets. It is worth mentioning that K{\oe}nigs showed in \cite{Koenigs} that planar nets with equal Laplace invariants can be understood locally by the condition that six lines are tangent to a conic.

\subsection{Conics associated to the strips}

\begin{thm}\label{thm: gridstructure}
Let $k_0, k_1, \ldots, k_m$ and $l_0, l_1, \ldots, l_n$ $(m,n \in \N_{\geq 2})$ be generic lines in the projective plane. Assume that the quadrilaterals $Q_{i,j}:= \Box (k_{i-1}, l_{j-1}, k_i, l_j)$ are equipped with inscribed conics $\mathcal{C}_{i,j}$ such that, for any two neighbouring quadrilaterals, the inscribed conics are touching. Let  $k_{i-1}L_{j-1,j}$, $k_iL_{j-1,j}$, $K_{i-1,i}l_{j-1}$, $K_{i-1,i}l_j$ be the corresponding tangency points, labelled by their tangent lines and by the strips of quadrilaterals. Then, all of the gridlines $k_0, k_1, \ldots, k_m$ and $l_0, l_1, \ldots, l_n$ are tangent to a conic $\mathcal{C}$. For any fixed $i$, the points $\{K_{i-1,i}l_j\}_{j=0, 1, \ldots, n}$ are contained in a conic $\mathcal{A}_{i}$ that has double contact with the conic $\mathcal{C}$. For any fixed $j$, the points $\{k_{i}L_{j-1,j}\}_{i=0, 1, \ldots, m}$ are contained in a conic $\mathcal{B}_{j}$ that has double contact with the conic $\mathcal{C}$. (See Figure~\ref{figure: gjerj}.) Among the $1$-parameter family of touching inscribed conics, there is one instance such that all of the conics $\{\mathcal{A}_i\}_{i=1, \ldots, m}$ and $\{\mathcal{B}_j\}_{j = 1, \ldots, n}$ are double lines. (See Figure~\ref{figure: fsafbabf}.)
\end{thm}

\begin{figure}[htbp]
\begin{center}
  \begin{subfigure}[t]{0.45\textwidth}
    \includegraphics[width=\textwidth]{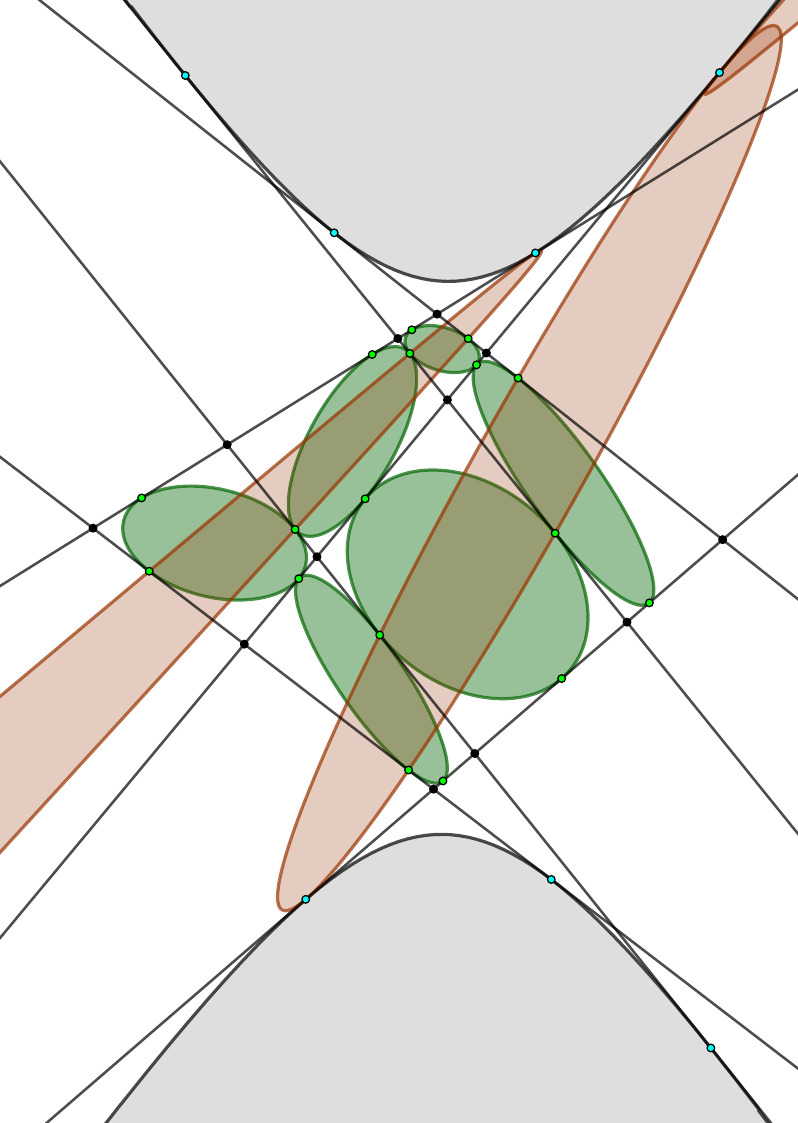}     
 \end{subfigure}
\hspace{0.9cm}%
 \begin{subfigure}[t]{0.45\textwidth}
    \includegraphics[width=\textwidth]{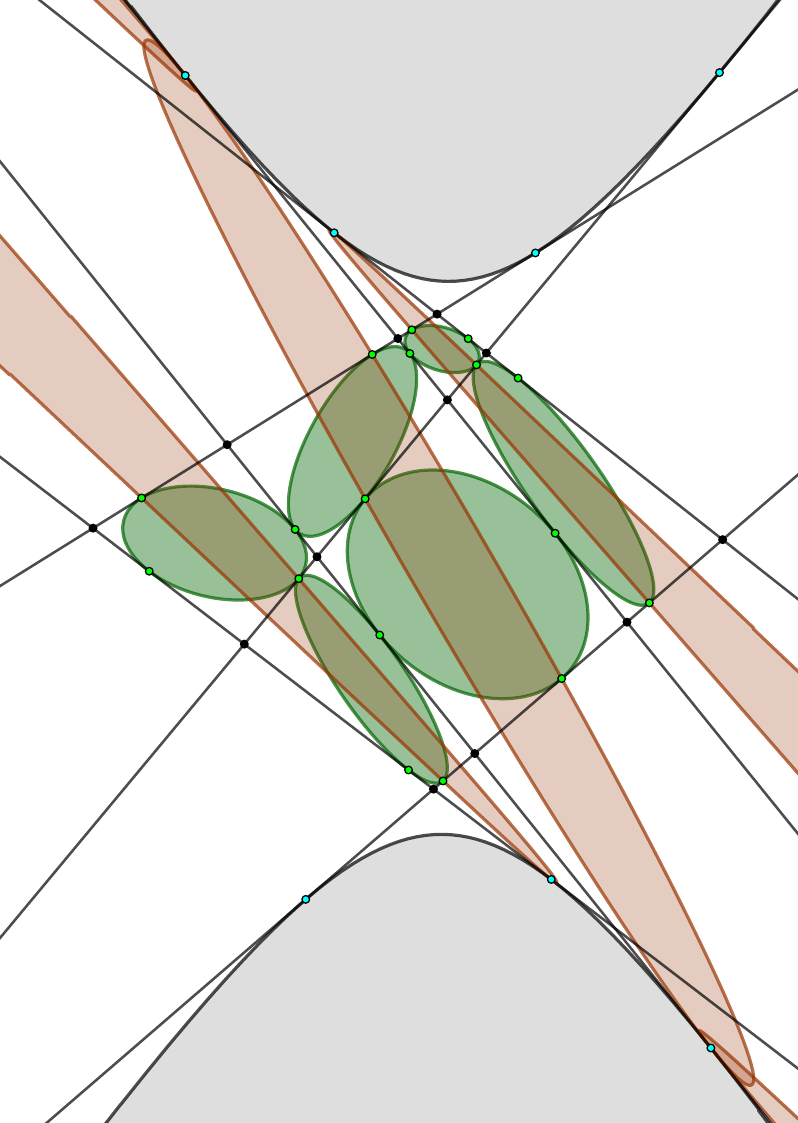}
  \end{subfigure}
  \end{center}
 \caption{A $2 \times 3$ grid of quadrilaterals with touching inscribed conics. The lines are tangent to a conic $\mathcal{C}$. The tangency points of the touching inscribed conics are contained in conics that have double contact with $\mathcal{C}$.} 
  \label{figure: gjerj}
\end{figure}

We start with Lemma~\ref{lem: 2ndtangent}, which will be used in the proof of Theorem~\ref{thm: gridstructure}.

\begin{lem}\label{lem: 2ndtangent}
Let $p_{([u],[v])}$, $p_{([w],[x])}$, $q_{([x],[u])}$, $q_{([v],[w])}$ be the tangency points of a conic that is inscribed in the quadrilateral $\Box([u],[v],[w],[x])$ in $\P$. Let $p_{([u],[v])}^\ast$, $p_{([w],[x])}^\ast$, $q_{([x],[u])}^\ast$, $q_{([v],[w])}^\ast$  be the tangency points of another inscribed conic. The tangency points are labelled by their tangent lines. Then, there exists a unique conic containing the points $p_{([w],[x])}$, $p_{([u],[v])}$, $q_{([x],[u])}^\ast$, $q_{([v],[w])}^\ast$ and with the tangent lines $([w], [x])$, $([u], [v])$. Symmetrically, there exists a unique conic containing the points $q_{([x],[u])}$, $q_{([v],[w])}$, $p_{([u],[v])}^\ast$, $p_{([w],[x])}^\ast$ and with the tangent lines $([x], [u])$, $([v], [w])$.
\end{lem}

\begin{figure}[htbp]
\begin{center}
    \includegraphics[width=0.95\textwidth]{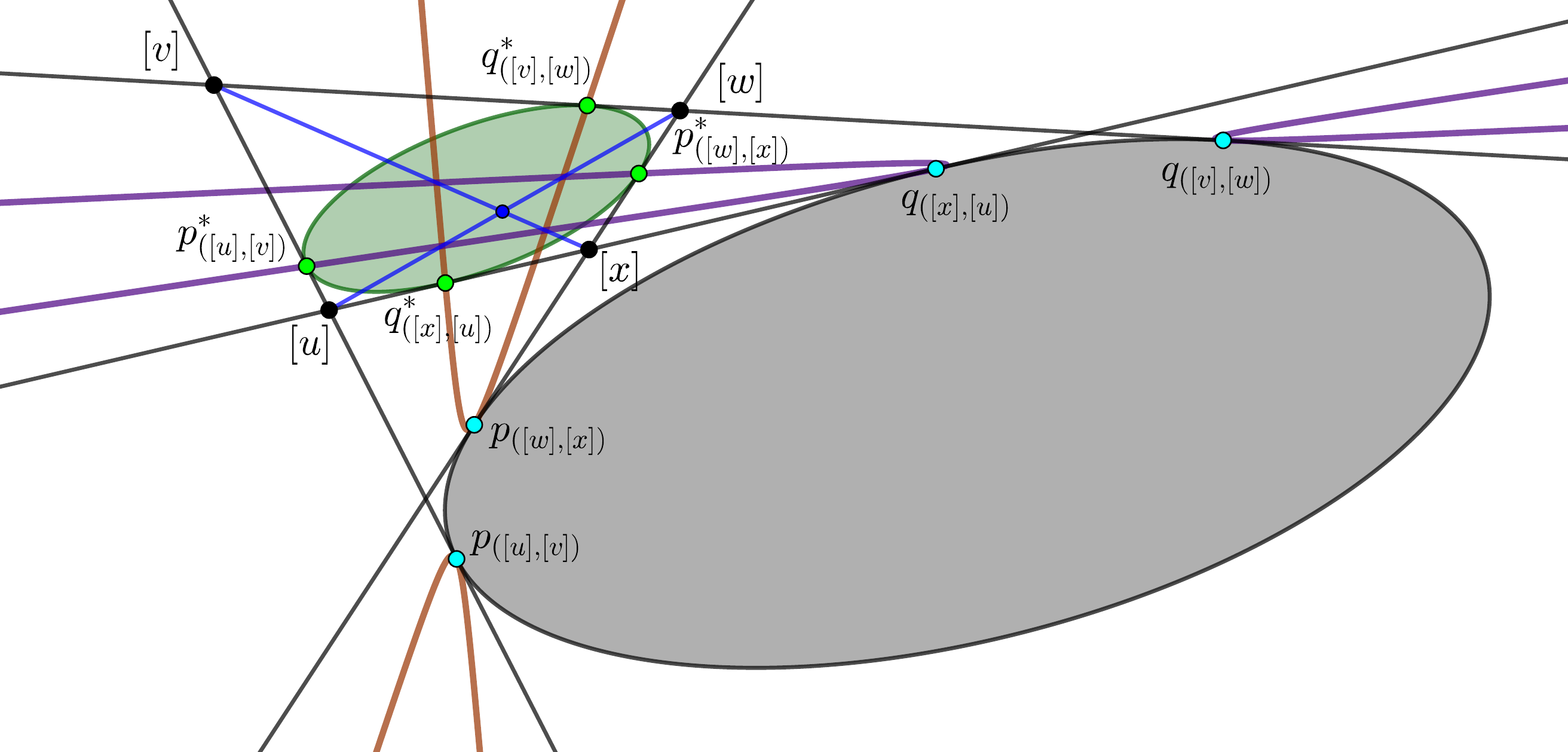}
\end{center}
\caption{There exists a conic containing the points $p_{([w],[x])}$, $p_{([u],[v])}$, $q_{([x],[u])}^\ast$, $q_{([v],[w])}^\ast$ and with the tangent lines $([w], [x])$, $([u], [v])$. Symmetrically, there exists a conic containing the points $q_{([x],[u])}$, $q_{([v],[w])}$, $p_{([u],[v])}^\ast$, $p_{([w],[x])}^\ast$ and with the tangent lines $([x], [u])$, $([v], [w])$.}
       \label{figure: fgrgagaeg}
\end{figure}

\begin{proof}
By Theorem~\ref{thm: inquad}, it is possible to choose the representative vectors $u$, $v$, $w$, $x$ for the vertices so that $p_{([u],[v])}^\ast =[u+v]$, $q_{([v],[w])}^\ast = [v+w]$, $p_{([w],[x])}^\ast =  [w+x]$, $q_{([x],[u])}^\ast = [x+u]$ and so that $[u+w] = [v+x]$ is the intersection of the two diagonals. By Theorem~\ref{thm: inquad}, for any non-zero $\l \in \R$, the points $[u+\l v]$, $[\l v +w]$, $[w+ \l x]$, $[\l x +u]$ are the tangency points of an inscribed conic. Choose the non-zero scalar $\l \in \R$ so that $p_{([u],[v])} = [u+\l v]$, $q_{([v],[w])}= [\l v +w]$, $p_{([w],[x])} = [w+ \l x]$, $q_{([x],[u])} = [\l x +u]$. To prove the lemma, it suffices to show that there exists a conic $\conic{\psi}$ containing the points $[u+\l v]$, $[v+w]$, $[w+\l x]$, $[x+u]$ and such that that $([u], [v])$ and $([w], [x])$ are tangent lines.

Define a non-zero symmetric bilinear form $\psi : \R^3 \times \R^3 \to \R$ by the following system of equations on the basis $u$, $v$, $w \in \R^3$.
\begin{align*}
\psi(u,u) =1   \qquad \psi (v,v) = \frac{1}{\l^2}  \qquad \psi(w,w) = 1\\ 
\psi(u,v)= \frac{-1}{\l}  \qquad \psi(u,w) = \frac{\l^2 + 2\l^2\c + 2\l +1}{-2\l^2 \c}  \qquad \psi(v,w) = \frac{\l^2+1}{-2\l^2}
\end{align*}

The identity $[u+w]=[v+x]$ implies $v+x = \c( u + w)$ for some non-zero scalar $\c \in \R$. By substituting $x = \c u - v + \c w$, the following identities are easily verified.
\begin{align*}
 \psi(u+\l v, u) = \psi(u+\l v, v) = \psi(w+\l x , w) = \psi(w+\l x, x)=0 \\ \psi(u+ \l v, u+\l v) = \psi (v+w,v+w) = \psi(w+\l x,w + \l x) = \psi(x+u, x+u)=0 
\end{align*}
The conic determined by $\psi$ contains the points $[u+\l v]$, $[v+w]$, $[w+\l x]$, $[x+u]$. The lines $([u], [v])$ and $([w],[x])$ are contained in the polars of $[u+\l v]$ and  $[w+\l x]$, respectively,
\end{proof}

\begin{figure}[htbp] 
\[\includegraphics[width=0.59\textwidth]{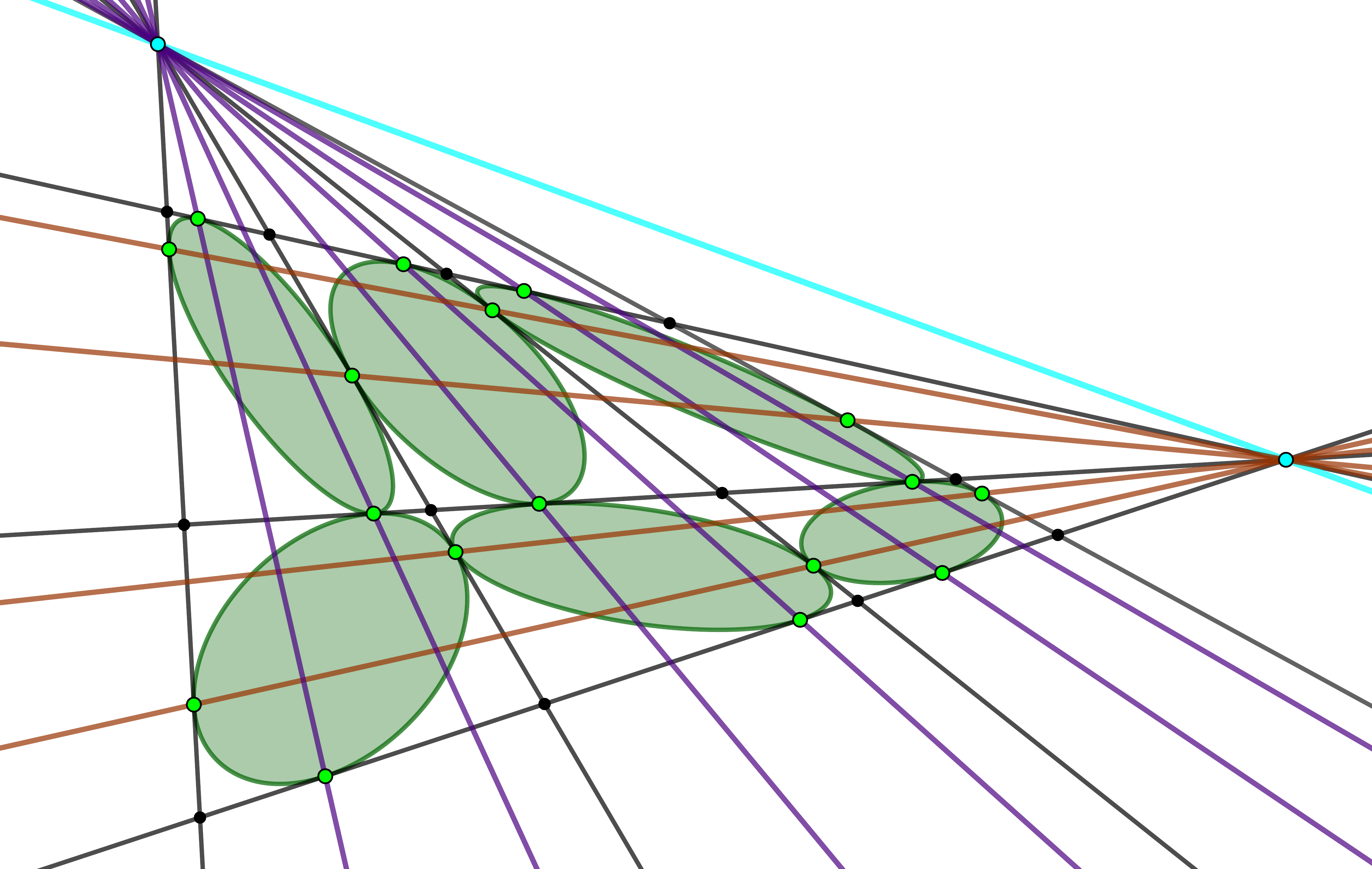}\]
\caption{A non-generic $3\times 2$ grid of quadrilaterals with touching inscribed conics. The corresponding common conic is degenerate. The lines $k_i$ are concurrent and so too are the lines $l_j$. The double contact conics $\mathcal{A}_i$ and $\mathcal{B}_j$ are also degenerate. They are pairs of lines passing through the concurrency points.} 
\label{figure: ewfewf}
\end{figure}

\begin{proof}[Proof of Theorem~\ref{thm: gridstructure}]
Each quadrilateral $Q_{i,j}$ is equipped with an inscribed conic $\mathcal{C}_{i,j}$ such that, for any two neighbouring quadrilaterals, the inscribed conics are touching. By Corollary~\ref{cor: tangentgridlines}, the lines $k_0, k_1, \ldots, k_m$ and $l_0, l_1, \ldots, l_n$ are tangent to a non-degenerate conic $\mathcal{C}$. Let $p_0, p_1, \ldots, p_m$ and $q_0, q_1, \ldots, q_n$ be the respective tangency points. For any fixed $i \in \{1, \dots, m\}$, consider the strip $K_{i-1,i}$ of quadrilaterals $\{Q_{i,j}\}_{j=1,\ldots,n}$. For each quadrilateral $Q_{i,j}$ in $K_{i-1,i}$, Lemma~\ref{lem: 2ndtangent} determines a conic, say $\mathcal{D}_j$, containing the points $p_{i-1}, p_i, K_{i-1,i}l_{j-1}, K_{i-1,i}l_j$ and with the tangent lines $k_{i-1}, k_i$. The conics $\mathcal{D}_j$ and $\mathcal{D}_{j+1}$ are equal because $p_{i-1}, p_i, K_{i-1,i}l_{j}$ are common points and the lines $k_{i-1}, k_i$ are common tangents. Thus, $\mathcal{A}_i$ is the conic $\mathcal{D}_1 = \ldots= \mathcal{D}_n$. Therefore, the conics $\{\mathcal{A}_i\}_{i=1, \ldots, m}$ exist and, symmetrically, the conics $\{\mathcal{B}_j\}_{j=1,\dots,n}$ also exist. By Theorem~\ref{thm: mainthm} and Corollary~\ref{cor: tangentgridlines}, there is one instance of touching inscribed conics such that all of the conics $\{\mathcal{A}_i\}_{i=1, \ldots, m}$ and $\{\mathcal{B}_j\}_{j = 1, \ldots, n}$ are double lines.
\end{proof}

\subsection{Incircular nets and billiards in conics}\label{subsection: billiards}

\emph{Incircular nets} are line grids with quadrilaterals with inscribed circles. The following characterisation of incircular nets can be found in \cite{IC}[Definition 2.3].

\begin{defn}\label{defn: IC}
Let $a_0, a_1, \ldots, a_m$ and $b_0, b_1, \ldots, b_n$ be lines in the Euclidean plane. The $m \times n$ grid of quadrilaterals $\Box(a_{i-1}, b_{j-1}, a_i, b_{j})$ is an \emph{incircular net} if and only if the following conditions are satisfied.
\begin{enumerate}[label=(\roman*)]
\item The lines $a_0, a_1, \ldots, a_m$ and $b_0, \ldots, b_n$ are tangent to a conic $\mathcal{C}$.
\item The points $a_{i-1} \cap a_i$ and $b_{j-1} \cap b_j$ are contained in a conic $\mathcal{D}$ that is confocal with $\mathcal{C}$.
\end{enumerate}
\end{defn} 

The lines $a_0, a_1, \ldots, a_m$ and $b_0, b_1, \ldots, b_n$ are the lines of two billiards in the conic $\mathcal{D}$ that have the same confocal caustic $\mathcal{C}$. Billiards in conics have caustics that are confocal conics \cite{Tabachnikov}.

The gridlines of any incircular net are tangent to a conic. Therefore, by Corollary~\ref{cor: tangentgridlines}, incircular nets are grids of quadrilaterals that admit a $1$-parameter family of touching inscribed conics. However, the inscribed circles of incircular nets are not touching inscribed conics. (See Figure~\ref{figure: ICnetwithTC}.)

\begin{thm}
For any incircular net, there is a dual grid of quadrilaterals that has a $1$-parameter family of touching inscribed conics. The vertices of the dual grid are the centres of the circles of the incircular net. The lines of the dual grid are angle bisector lines of the incircular net. (See Figure~\ref{figure: ICnetwithTC}.)
\end{thm}

\begin{proof}
Any incircular net determines two billiards $p_0, p_1, \ldots, p_m$ and $q_0, q_1, \ldots, q_n$ that are inscribed in a conic $\mathcal{D}$ and that have the same confocal caustic $\mathcal{C}$. Let $k_0, \ldots, k_m$ and $l_0, \ldots, l_n$ be the tangent lines of $\mathcal{D}$ at the points $p_0, p_1, \ldots, p_m$ and $q_0, q_1, \ldots, q_n$, respectively. By Corollary~\ref{cor: tangentgridlines}, the $m \times n$ grid of quadrilaterals $\Box(k_{i-1}, l_{j-1}, k_i, l_{j})$ admits a $1$-parameter family of touching inscribed conics. The billiard reflection law ensures that the tangent line of $\mathcal{D}$ at $p_i$ is an angle bisector of the lines $(p_{i-1}, p_i)$ and $(p_i, p_{i+1})$. Likewise, the tangent line of $\mathcal{D}$ at $q_j$ is an angle bisector of the lines $(q_{j-1}, q_j)$ and $(q_j, q_{j+1})$. By the Graves-Chasles theorem (see for example \cite{IC}), there is a circle that is tangent to the four dotted lines and that is centred at the intersection point of the tangent lines of $\mathcal{D}$ at $p_i$ and $q_j$. 
\end{proof}

\begin{figure}[htbp]
\begin{center}
\includegraphics[width =0.95\textwidth]{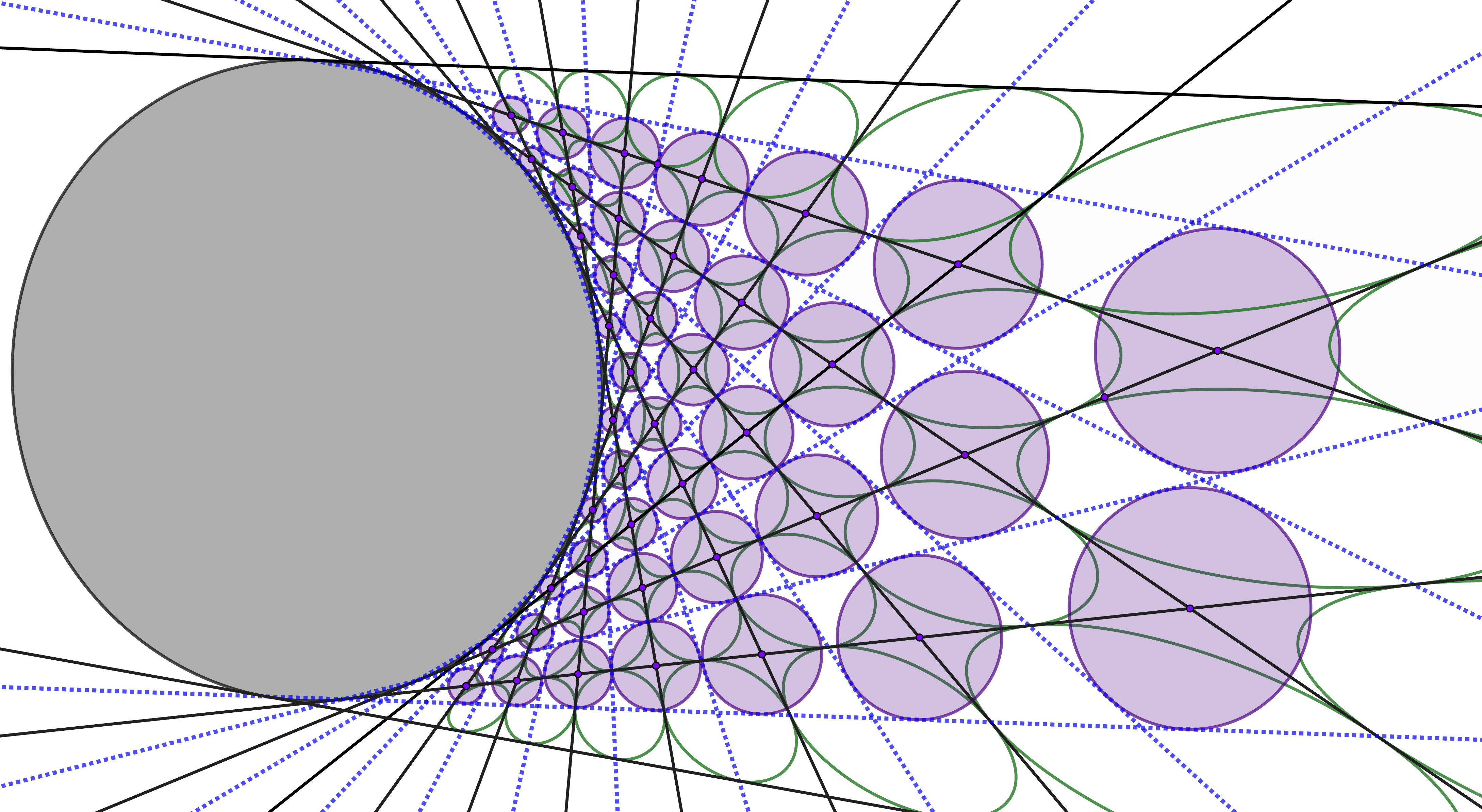}
\end{center}
\caption{An incircular net. The circles and the touching conics are inscribed in combinatorially dual line grids. The line grid with touching inscribed conics is given by the lines passing through the centres of the circles.}
\label{figure: ICnetwithTC}
\end{figure}

\begin{figure}[htbp!]
\begin{center}
\includegraphics[width =0.75\textwidth]{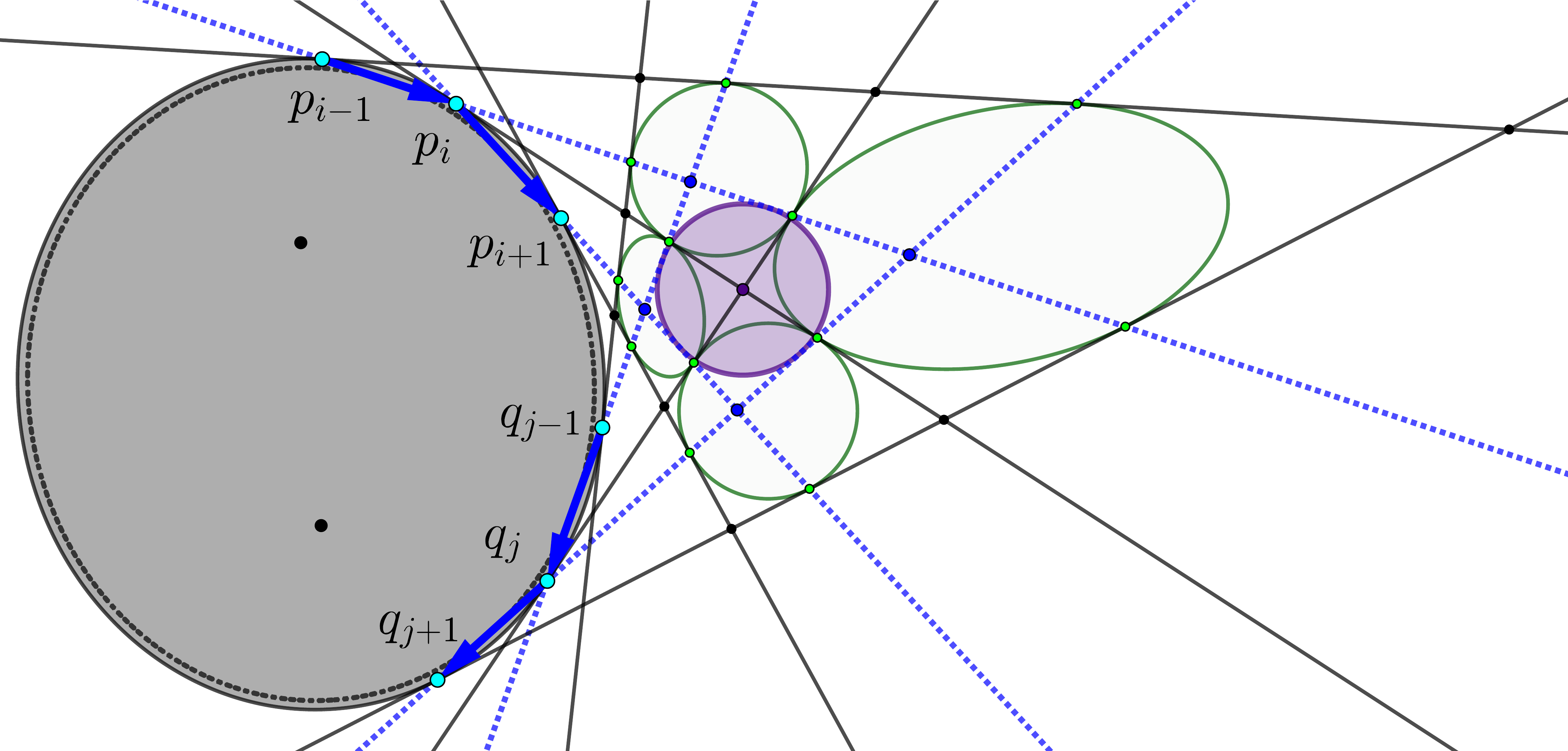}
\end{center}
\caption{The two polygonal chains $p_{i-1}, p_i$, $p_{i+1}$ and $q_{j-1}, q_j, q_{j+1}$ are billiards that are inscribed in a conic $\mathcal{D}$. Suppose that the two billiards have the same confocal caustic. Then, by the Graves-Chasles theorem, there exists a circle that is tangent to the four dotted lines. The centre of the circle is the intersection point of the tangent lines of $\mathcal{D}$ at $p_i$ and $q_j$. These tangent lines generate a line grid with touching inscribed conics.}
\label{figure: extwobilliards}
\end{figure}

\end{document}